\documentclass[12pt]{amsart}

\usepackage{amsfonts,amssymb,stmaryrd,amscd,amsmath,latexsym,amsbsy}

\newtheorem{theorem}{Theorem}[section]
\newtheorem{lemma}[theorem]{Lemma}
\newtheorem{proposition}[theorem]{Proposition}
\newtheorem{corollary}[theorem]{Corollary}
\theoremstyle{definition}
\newtheorem{definition}[theorem]{Definition}

\newtheorem{example}[theorem]{Example}
\newtheorem{question}[theorem]{Question}
\newtheorem{conjecture}[theorem]{Conjecture}
\newtheorem{remark}[theorem]{Remark}

\newcommand{\q}{{\bold q}}

\newcommand{\Tr}{\text{Tr}}

\newcommand{\Hom}{\text{Hom}}

\newcommand{\Id}{\text{Id}}
\newcommand{\Aut}{\text{Aut}}

\newcommand{\Ind}{\text{Ind}}
\newcommand{\Rep}{\text{Rep}}

\newcommand{\g}{\mathfrak{g}}

\newcommand{\kk}{{\bold k}}

\newcommand{\C}{\mathcal{C}}
\newcommand{\ot}{\otimes}

\newcommand{\ben}{\begin{enumerate}}
\newcommand{\een}{\end{enumerate}}

\newcommand{\be}{{\bf 1}}

\theoremstyle{plain}

\newtheorem*{sol}{Solution}

\theoremstyle{definition}

\theoremstyle{remark}

\newcommand{\solu}[1]{\begin{sol}{\bf (\ref{#1})}}

\def\g{\mathfrak{g}}

\def\C{\mathcal{C}}
\def\D{\mathcal{D}}

\def\Aut{\mathop{\mathrm{Aut}}\nolimits}

\def\N{\mathcal{N}}

\def\Hom{\mathrm{Hom}}

\def\Vec{\mathrm{Vec}}

\def\Ver{\mathrm{Ver}}
\def\k{\mathbf{k}}

\def\Rep{\mathop{\mathrm{Rep}}\nolimits}

\begin{document}

\title{On semisimplification of tensor categories}

\author{Pavel Etingof}
\address{Department of Mathematics, Massachusetts Institute of Technology,
Cambridge, MA 02139, USA}
\email{etingof@math.mit.edu}
\author{Victor Ostrik}
\address{Department of Mathematics,
University of Oregon, Eugene, OR 97403, USA}
\address{Laboratory of Algebraic Geometry,
National Research University Higher School of Economics, Moscow, Russia}
\email{vostrik@uoregon.edu}

\maketitle 

\begin{abstract} We develop the theory of semisimplifications of tensor categories defined by Barrett and Westbury. In particular, we compute the semisimplification of the category of representations of a finite group in characteristic $p$ in terms of representations of the normalizer of its Sylow 
$p$-subgroup. This allows us to compute the semisimplification of the representation category of the symmetric group $S_{n+p}$ in characteristic $p$, where $0\le n\le p-1$, and of the Deligne category $\underline{\Rep}^{\rm ab}S_t$, where $t\in \Bbb N$. We also compute the semisimplification of the category of representations of the Kac-De Concini quantum group of the Borel subalgebra of $\mathfrak{sl}_2$. We also study tensor functors between Verlinde categories of semisimple algebraic groups arising from the semisimplification construction, and objects of finite type in categories of modular representations of finite groups (i.e., objects generating a fusion category in the semisimplification). Finally, we determine the semisimplifications of the tilting categories of $GL(n)$, $SL(n)$ and $PGL(n)$ in characteristic $2$. In the appendix, we classify categorifications of the Grothendieck ring of representations of $SO(3)$ and its truncations. 
\end{abstract} 

\vskip .05in

\centerline{\bf To Sasha Beilinson and Vitya Ginzburg}
\centerline{\bf on their 60th birthdays
with admiration} 

\section{Introduction} 

The notion of the {\it semisimplification} of a spherical tensor category was introduced in \cite{BW}, 
although in the context of algebraic geometry it can be traced back to the notion
of numerical equivalence of cycles in the theory of motives, see e.g. \cite{Ja}. More generally,
various {\em adequate equivalence relations} in the same theory can be considered
as examples of {\em tensor ideals} in the symmetric tensor category of Chow motives.

Recall that a morphism $f:X\to Y$ in a spherical tensor category $\C$ over a field $\bold k$ 
is called {\it negligible} if for any morphism $g: Y\to X$, one has $\Tr(f\circ g)=0$. One can show that the collection ${\mathcal N}$ of negligible morphisms is a tensor ideal, thus one can define an additive monoidal category $\overline{\C}:=\C/{\mathcal N}$. One can show that $\overline{\C}$ is, in fact, semisimple abelian, with simple objects being the indecomposable objects of $\C$ of nonzero dimension, and it is called the {\it semisimplification} of $\C$. Moreover, this definition can be generalized to pivotal categories in which the left and right dimension of indecomposables vanish simultaneously, and even to Karoubian (not necessarily abelian) monoidal categories in which the trace of a nilpotent endomorphism is zero. 

The semisimplification construction is a rich source of semisimple tensor categories. In the simplest cases, when the classification of indecomposables in $\mathcal{C}$ is tame, the semisimplification can be described explicitly. Admittedly, this happens rather rarely: most of the time the classification of indecomposables is wild, and the corresponding semisimplified category $\overline{\C}$ is somewhat unmanageable, i.e., may have uncountably many simple objects even if $\C$ is finite (e.g., this happens already for $\C=\Rep_{\bold k}((\Bbb Z/p)^2)$, where $\bold k$ is an uncountable field with ${\rm char}(\bold k)=p>2$). However, in this case we may consider the tensor subcategory of $\mathcal{\C}$ generated by a given object $X$, which is  much more manageable (in particular, always has a finite or countable set of isomorphism classes of simple objects); in particular, it is an interesting question when this subcategory is fusion (i.e., has finitely many simple objects), and what it looks like in this case. 

The goal of this paper is to develop a number of tools for studying semisimplifications of tensor categories, and to apply them 
to compute the semisimplifications and their tensor subcategories generated by particular objects in a number of specific examples. 

Specifically, in Section 2 we review the basic theory of tensor ideals and semisimplifications. 

In Section 3, we give some general results about semisimplications. In particular, we discuss semisimplifications of Tannakian categories in characteristic zero, reductive envelopes of algebraic groups and the generalized Jacobson-Morozov Lemma (following Andr\'e and Kahn), compatibility of semisimplification with equivariantization and with surjective tensor functors. 

In Section 4, we use classical results of modular representation theory (the Green correspondence) to show that the semisimplification of the category $\Rep G$ of representations of a finite group $G$ in characteristic $p>0$ is naturally equivalent to that of the normalizer of its $p$-Sylow subgroup, and compute the semisimplification of $\Rep G$ when the Sylow subgroup is cyclic of order $p$ (in particular for $G=S_{n+p}$ with $0\le n<p$). We then use this result and the work of Harman  
to compute the semisimplification of the abelian envelope of the Deligne category $\underline{\Rep}^{\rm ab}(S_n)$. 

In Section 5 we compute the semisimplifications of some non-symmetric categories in characteristic zero, namely, the category of representations of the Kac-De Concini quantum group $U_q(\mathfrak{b})$, where $\mathfrak{b}$ is the Borel subalgebra  
of $\mathfrak{sl}_2$ when $q$ is generic and when $q$ is a root of unity. 

In Section 6, we study surjective tensor functors between Verlinde categories attached to simple algebraic groups in characteristic $p$; interesting examples of such functors, which are attached to pairs of simple algebraic groups $G\supset K$ where $K$ contains a regular unipotent element of $G$, are obtained from the semisimplification construction.  

In Section 7, we study objects of finite type in semisimplications of categories of group representations in characteristic $p$, i.e., objects generating fusion subcategories. We give a number of nontrivial examples of objects of finite type, and study the fusion categories they generate. 

In Section 8 we determine the semisimplifications of the tilting categories of $GL(n)$, $SL(n)$ and $PGL(n)$ in characteristic $2$.  

Finally, in the appendix we classify categorifications of the representation ring and Verlinde ring for $SO(3)$. This is used in Section 5. 

{\bf Acknowledgements.} The authors are grateful to D. Benson, I. Entova-Aizenbud, T. Heidersdorf, A. Kleshchev, D. Nakano, V. Serganova, and N. Snyder for useful discussions. The work of P.E. and V.O. was partially supported by the NSF grants DMS-1502244 and DMS-1702251.
The work of V.O. has also been funded by the Russian Academic
Excellence Project '5-100'.

\section{Preliminaries}

\subsection{Tensor ideals} Let $\kk$ be a field and
let $\C$ be a $\kk-$linear monoidal category. Recall that
a {\em tensor ideal} $I$ in $\C$ is a collection of subspaces $I(X,Y)\subset \Hom(X,Y)$ 
for all $X,Y\in \C$  such that for all $X,Y,Z,T \in \C$

(1) for $\alpha \in I(X,Y)$ and $\beta \in \Hom(Y,Z), \gamma \in \Hom(Z,X)$ we have
$\alpha \circ \gamma \in I(Z,Y)$ and $\beta \circ \alpha \in I(X,Z)$;

(2) for $\alpha \in I(X,Y), \beta \in \Hom(Z,T)$ we have $\alpha \ot \beta \in
I(X\ot Z, Y\ot T)$ and $\beta \ot \alpha \in I(Z\ot X, T\ot Y)$.

If $I$ is a tensor ideal in $\C$ then one can define a new $\kk-$linear monoidal category 
$\C'$ (the {\em quotient} of $\C$ by $I$) as follows: 
the objects of $\C'$ are the objects of $\C$; $\Hom_{\C'}(X,Y):=\Hom_\C(X,Y)/I(X,Y)$; the
composition of morphisms is the same as in $\C$ (note that condition (1) ensures that the
composition is well defined); the tensor product is the same as in $\C$ (well defined thanks to
condition (2)).

Moreover, the identity map on the objects and morphisms induces a canonical quotient monoidal functor
$\C \to \C'$. 

It is clear that if $\C$ is rigid, pivotal, spherical, braided, symmetric then so is $\C'$.

\subsection{Semisimplification of  a spherical tensor category} 

We recall the theory of semisimplifications of spherical tensor categories, due to Barrett and Westbury, \cite{BW}. We give proofs for reader's convenience. 

Let $\k$ be an algebraically closed field, and $\C$ be a spherical tensor category over $\k$ (see \cite{EGNO}, Subsection 4.7). 

\begin{definition} A morphism $f: X\to Y$ in $\C$ is called {\it negligible} if for any morphism $g: Y\to X$ one has $\Tr(f\circ g)=0$. 
\end{definition} 

\begin{lemma}\label{charneg} (\cite{B2}, Exercise 3(ii), Subsection 2.18) Let $X=\oplus_i X_i$ and $Y=\oplus_j Y_j$ be decompositions of $X,Y$ into indecomposable objects, 
and $f=\oplus_{i,j}f_{ij}$ be a morphism $X\to Y$, where $f_{ij}: X_i\to Y_j$. Then $f$ is negligible if and only if
for each $i,j$ either $\dim Y_j=0$ or $f_{ij}$ is not an isomorphism (equivalently, either $\dim X_i=0$ or $f_{ij}$ is not an isomorphism). 
\end{lemma} 

\begin{proof} First let us prove the lemma when $X,Y$ are indecomposable. If $f: X\to Y$ is not an isomorphism, then 
for any $g: Y\to X$, the morphism $f\circ g: Y\to Y$ 
is not an isomorphism, either; otherwise $f$ is injective (hence not surjective) and $X\cong {\rm Im}f\oplus {\rm Ker}g$, 
with both summands nonzero, giving a contradiction. Hence, $f\circ g$ is nilpotent and $\Tr(f\circ g)=0$. Also, if $f$ is an isomorphism (so $\dim X=\dim Y$)  then for any $g: Y\to X$, one has $f\circ g=\lambda{\rm Id}+h$, where $\lambda\in \bold k$ and $h: Y\to Y$ is nilpotent. Hence $\Tr(f\circ g)=\lambda \dim Y=\lambda \dim X$. If $\dim X=\dim Y=0$, this is always zero, while if $\dim X=\dim Y\ne 0$ then we can take $g=f^{-1}$ (so that $\lambda=1$), and 
$\Tr(f\circ g)=\dim Y\ne 0$, as desired. 
  
Now consider the general case. Suppose the condition of the lemma is satisfied, and $g: Y\to X$ is a morphism, $g=(g_{ji})$. Then 
$\Tr(f\circ g)=\sum_{i,j}\Tr(f_{ij}\circ g_{ji})$. If  either $\dim Y_j=0$ or $f_{ij}$ is not an isomorphism (equivalently, either $\dim X_i=0$ or $f_{ij}$ is not an isomorphism) for all $i,j$ then by the indecomposable case, $\Tr(f_{ij}\circ g_{ji})=0$ for all $i,j$, hence $\Tr(f\circ g)=0$. 
However, if for some $i,j$ this condition is violated, then we can take $g_{ji}=f_{ij}^{-1}$ and $g_{pq}=0$ for $(p,q)\ne (i,j)$, 
so that $\Tr(f\circ g)=\dim X_i=\dim Y_j$. This implies the lemma. 
\end{proof} 

Let $\N(\C)$ be the collection of negligible morphisms of $\C$. 

\begin{lemma}\label{tensid} 
$\N(\C)$ is a tensor ideal in $\C$. 
\end{lemma} 

\begin{proof} It is clear that a linear combination of negligible morphisms is negligible. Also, it is easy to see that 
$f\circ a, b\circ f$ are negligible for any $a,b$ (when these compositions make sense). It remains to show that 
the tensor products $a\otimes f$ and $f\otimes b$ are negligible. Let us prove this for $a\otimes f$, where $a: Z\to T$; the case of $f\otimes b$ is similar. Let $g: T\otimes Y\to Z\otimes X$. Then $\Tr((a\otimes f)\circ g)=\Tr(f\circ g')$, where $g':=\Tr_T((a\otimes {\rm Id})\circ g)$. 
Hence  $\Tr((a\otimes f)\circ g)=0$ and $a\otimes f$ is negligible, as desired. 
\end{proof}

Thus we can define a spherical tensor category $\overline{\C}:=\C/\N(\C)$. 

\begin{proposition}\label{semisi} The category $\overline{\C}$ is a semisimple tensor category. The simple objects of 
$\overline{\C}$ are the indecomposable objects of $\C$ of nonzero dimension. 
\end{proposition} 

\begin{proof} It is clear that indecomposable objects of $\overline{\C}$ are images of indecomposable objects of $\C$. More precisely, if $X,Y\in \C$ are indecomposable then by Lemma \ref{charneg}, $\Hom_{\overline{\C}}(X,Y)=0$ if $X\ncong Y$ or $\dim X=0$ or $\dim Y=0$ (i.e., if $\dim X=0$, then $X=0$ in $\overline{\C}$), and $\dim \Hom_{\overline{\C}}(X,Y)=1$ if $X\cong Y$ and $\dim X\ne 0$. This implies the proposition. 
\end{proof} 

\begin{definition} The category $\overline{\C}$ is called the {\it semisimplification} of $\C$. 
\end{definition} 

Note that the category $\overline{\C}$ comes equipped with a natural monoidal functor ${\bold S}: \C\to \overline{\C}$, which we call the {\it semisimplification functor}. This functor, however, is not a tensor functor, since it is not left or right exact, in general. We will denote the image ${\bold S}(X)$ of an object $X$ under this functor by $\overline{X}$. 

\subsection{Generalization to pivotal Karoubian categories} 

The above results generalize to pivotal tensor categories (\cite{EGNO}, Subsection 4.7)
such that $\dim^L X=0$ if and only if $\dim^R X=0$ for 
any indecomposable object $X\in \C$ (an example of such a category which is not spherical is the category of representations of the Taft Hopf algebra). Namely, in such a category, for any endomorphism $h: X\to X$ of an indecomposable object $X$, 
one has $\Tr^L(h)=0$ if and only if $\Tr^R(h)=0$. Thus, if $f: X\to Y$ is a morphism between arbitrary objects of $\C$, then the condition that for any $g:Y\to X$, one has $\Tr^L(f\circ g)=0$ is equivalent to the condition that for any $g:Y\to X$, one has $\Tr^R(f\circ g)=0$. 
One then defines $f$ to be negligible if any of these two equivalent conditions is satisfied. Then  Lemma \ref{charneg}, Lemma \ref{tensid} and Proposition \ref{semisi} generalize verbatim, with analogous proofs. 

Moreover, the above results also extend to the case when $\C$ is a Karoubian rigid monoidal category in which the trace of a nilpotent endomorphism is zero, a necessary condition for $\C$ to be embeddable into an abelian tensor category.\footnote{Note that this condition is not necessarily satisfied: e.g. if ${\rm char}(\k)=p$, $t\in \k$, and $\underline{\Rep}_\k(S_t)$ is the Karoubian Deligne category of representations of $S_t$ (\cite{EGNO}, Subsection 9.12) then this property holds only if $t\in \Bbb F_p\subset \k$; namely, if $\sigma$ is the cyclic permutation on $X^{\otimes p}$, where $X$ is the tautological object, then $(1-\sigma)^p=0$ but 
$\Tr(1-\sigma)=t^p-t$.} For instance, the well-known construction of the fusion categories attached to a simple Lie algebra $\g$ (in characteristic zero or $p$ bigger than the Coxeter number), \cite{EGNO}, Subsection 8.18.2, starts with the category of tilting modules for the corresponding (quantum) group (which is Karoubian), and takes a quotient by the tensor ideal of negligible morphisms. Note that in this special case negligible morphisms happen to be those that factor through negligible objects (i.e., direct sums of simple objects of dimension $0$); this is not the case in general (e.g., for $\Rep_\k(\Bbb Z/p)$). 

To summarize, we have the following result. Let $\C$ be a pivotal category, let $\dim^L(X):=\Tr^L({\rm Id}_X)$, 
$\dim^R(X):=\Tr^R({\rm Id}_X)$ for $X\in \C$, and call a morphism 
$f: X\to Y$ negligible if for any $g: Y\to X$ one has $\Tr^L(f\circ g)=0$. 

\begin{theorem}\label{karpiv} Let $\C$ be a $\kk-$linear Karoubian rigid monoidal category such that all morphisms
spaces are finite dimensional.\footnote{Note that any Karoubian linear category with finite dimensional morphism spaces satisfies the Krull-Schmidt theorem, which says that any object has a unique decomposition into a direct sum of indecomposables (up to a non-unique isomorphism); for this reason, such categories are  sometimes called Krull-Schmidt categories.} Assume that $\C$ is equipped with a pivotal structure such that 

(1) the left trace $\Tr^L$ of any nilpotent endomorphism is zero; 

(2) $\dim^L X=0$ if and only if $\dim^R X=0$ for an indecomposable $X\in \C$.

Then negligible morphisms are characterized as in Lemma \ref{charneg} and 
form a tensor ideal ${\mathcal N}(\C)$. Moreover, 
$\C/\N(\C)$ is a semisimple tensor category, whose simple objects are the indecomposable objects 
of $\C$ of nonzero dimension. 
\end{theorem}

\begin{proof} First of all, (1) implies that the right trace of any nilpotent endomorphism in $\C$ is zero, 
since $\Tr^L(f)=\Tr^R(f^*)$, see \cite{EGNO}, Proposition 4.7.3. 

Hence, for an endomorphism $h: X\to X$, $\Tr^L(h)=0$ if and only if $\Tr^R(h)=0$. 
Indeed, by decomposing $X$ into generalized eigenobjects of $h$, we may assume that $h=\lambda{\rm Id}+h_0$, where $h_0$ is nilpotent. 
Then $\Tr^L(h)=\lambda \dim^L X$ and $\Tr^R(h)=\lambda \dim^R X$ (as $\Tr^L(h_0)=\Tr^R(h_0)=0$), so our claim follows from (2). 

The rest of the proof is parallel to the spherical abelian case. 
\end{proof} 

\begin{example}\label{ex1} 1. If $\C$ is semisimple, then $\overline{\C}\cong\C$. Moreover, in this case for any tensor category $\D$ one has 
$\overline{\C\boxtimes \D}\cong\C\boxtimes \overline{\D}$. 

2. If ${\rm char}(\bold k)=p>0$ and $\C=\Rep_{\bold k} (\Bbb Z/p)$ then $\overline{\C}$ is the Verlinde category ${\rm Ver}_p$ introduced by Gelfand-Kazhdan and Georgiev-Mathieu, 
see \cite{O} and references therein.  

3. Let ${\rm char}(\k)=0$ and $\C=\Rep GL(n|1)$, $n\ge 1$. Then $\overline{\C}=\Rep (GL(n-1)\times GL(1)\times GL(1))\boxtimes {\rm Supervec}$, where ${\rm Supervec}$ is the category of supervector spaces, see \cite{H}, Theorem 4.13. 

4. Let $G=({\mathbb Z}/2{\mathbb Z})^2$ and ${\rm char}(\k)=2$.
Then it is well known that indecomposable representations of $G$ over $\k$ of non-zero mod $2$ (i.e. odd)
dimension are precisely $\Omega^n(\be), n\in {\mathbb Z}$, where $\Omega$ is the Heller
shift operator, see e.g. \cite[Theorem 4.3.3]{B1}. Also one deduces from \cite[Corollary 3.1.6]{B1} that
$$\Omega^n(\be)\otimes \Omega^m(\be)\simeq \Omega^{n+m}(\be)\oplus \mbox{a projective module}.
$$
Thus $\overline{\Rep_\k(G)}=\Vec_{\mathbb Z} =\Rep GL(1)$.
\end{example} 

\begin{remark} 1. It is clear that if $\C$ is symmetric or braided, then so is $\overline{\C}$ and the functor $\bold S$.

2. If $\C$ is finite then $\overline{\C}$ may be infinite (see Example \ref{ex1}(4)), and can, in fact, be unmanageably large, since
the problem of classifying indecomposable objects in finite abelian categories is often wild (in fact, this is already so for $\Rep_{\bold k}(\Bbb Z/p)^2$, where ${\rm char}(\bold k)=p>2$).  
\end{remark} 

\begin{remark}\label{remhopf} 1. Let $\C=\Rep H$, where $H$ is a finite dimensional Hopf algebra over a field $\bold k$ of characteristic zero. 
Then condition (2) of Theorem \ref{karpiv} (that $\dim^L X=0$ if and only if $\dim^R X=0$) holds for any pivotal structure. Indeed, we may assume that $\k=\Bbb C$. A pivotal structure on $\Rep H$ is given by a grouplike element $g\in H$ such that $gxg^{-1}=S^2(x)$ for $x\in H$, and $\dim^L X=\Tr_X(g)$, $\dim^RX=\Tr_X(g^{-1})$. But $g$ has finite order, so the eigenvalues of $g$ are roots of unity, hence $\dim^R X=\overline{\dim^L X}$, as desired. 
We expect that the same holds for any finite tensor category over a field of characteristic zero. 

However, the above condition can be violated for categories of finite dimensional modules or comodules 
over an infinite dimensional Hopf algebra. For example, let $\C$ be the category of finite dimensional representations 
of $U_q(\mathfrak{b})$, $q\in \Bbb C^\times$, where $\mathfrak{b}\subset \mathfrak{sl}_3$ is a Borel subalgebra. Recall that a pivotal structure on $\C$ 
is defined by the element $K=q^{2\rho}$. Let $X$ be the $U_q(\mathfrak{b})$ subrepresentation of the adjoint representation of $U_q(\mathfrak{sl}_3)$ (with highest weight $\alpha_1+\alpha_2$) spanned by the vectors whose weights are positive roots. Then $\dim^L X=2q^2+q^4$ and $\dim^R X=2q^{-2}+q^{-4}$. So if $q^2=-2$ then $\dim^L X=0$ but $\dim^R X=-3/4\ne 0$. 

The same happens in characteristic $p$, even for a finite dimensional Hopf algebra. Namely, we can take the same example. Note that $q^2=-2$ 
is then a root of unity (or some order dividing $p-1$), so one may replace $U_q(\mathfrak b)$ with the corresponding small quantum group $\mathfrak{u}_q(\mathfrak{b})$.  

2. Condition (1) of Theorem \ref{karpiv} holds true if $\C$ is an abelian tensor category, since
the quantum trace is additive on exact sequences, see e.g. \cite[Proposition 4.7.5]{EGNO}.
Moreover, assume that there exists a pivotal tensor functor $\C \to \D$, where the category $\D$
satisfies condition (1) of Theorem \ref{karpiv} (e.g., $\D$ is abelian). Then obviously the category 
$\C$ also satisfies condition (1) of Theorem \ref{karpiv}. This observation was used by
U.~Jannsen to prove that the category of numerical motives is semisimple, see \cite{Ja}. 
Moreover, the assumption on finite dimensionality of morphism spaces in $\C$ in Theorem \ref{karpiv} can be dropped if there exists a pivotal monoidal functor $F:\C \to \D'$, where all morphism
spaces in $\D'$ are finite dimensional, since the tensor ideal of morphisms sent by $F$ to
zero consists of negligible morphisms, which implies finite dimensionality of morphism spaces
in $\C/\N(\C)$. 

Here is an example of such a situation. Take any collection of morphisms in a symmetric tensor category $\D$, compute
some of relations between them, and define $\C$ to be the Karoubian envelope of the universal symmetric monoidal category 
generated by morphisms satisfying these relations. Then we have an obvious symmetric monoidal functor $\C\to \D$, hence the semisimplification of $\C$ is a semisimple symmetric tensor category. 
\end{remark} 

\section{General results on semisimplification of tensor categories}

\subsection{Splitting of the semisimplification functor for Tannakian categories in characteristic zero, reductive envelopes, and the Jacobson-Morozov lemma} 

For Tannakian categories in characteristic zero, Andr\'e and Kahn showed that the semisimplification functor $\bold S$ admits a splitting $\bold S^*$, and used it to show the existence and uniqueness (up to conjugation) of the reductive envelope of any affine proalgebraic group in characteristic zero. In this subsection we review this theory (cf. \cite{AK},\cite{S}). 

\begin{theorem} \label{splitting} (\cite{AK}, Theorem 1, Theorem 2) If ${\rm char}(\k)=0$ and $\C=\Rep G$ is a Tannakian category over $\k$ (where $G$ is an affine proalgebraic group over $\k$), then the functor $\bold S: \C\to \overline{\C}$ admits a splitting $\bold S^*: \overline{\C}\to {\C}$, a surjective tensor functor such that $\bold S^*(\overline X)\cong X$ for each indecomposable $X\in \C$, and ${\bold S}\circ \bold S^*\cong {\rm Id}$ as a symmetric tensor functor. 
\end{theorem} 

Now let $\C$ be as above and $F$ be the forgetful functor $\C\to \Vec$. Then $F\circ \bold S^*: \overline{\C}\to \Vec$ is a fiber functor, so by the Tannakian formalism (\cite{DM}), we have $\overline{\C}=\Rep \overline{G}$, where $\overline{G}:=\Aut(F\circ \bold S^*)$ is a reductive affine proalgebraic group, equipped with a homomorphism $\psi_G: G\to \overline{G}$ (defined up to conjugation in $\overline{G}$) giving rise to the functor $\bold S^*$. Moreover, since $\bold S^*$ is surjective, $\psi_G$ is an inclusion. 

\begin{definition} (\cite{AK}) The group $\overline{G}$ equipped with the homomorphism $\psi_G$ (defined up to conjugation) is called the 
{\it reductive envelope} of $G$. 
\end{definition} 

\begin{theorem} (\cite{AK}, Theorem 3, Theorem 4)\label{ak} The reductive envelope $\overline{G}$ enjoys the following universal property: 
If $\phi: G\to L$ is a homomorphism from $G$ to a reductive proalgebraic group $L$, then there exists a homomorphism 
$\overline{\phi}: \overline{G}\to L$ such that $\phi=\overline{\phi}\circ \psi_G$. Moreover, $\overline{\phi}$ is unique up to conjugation in $L$ 
by elements commuting with $\phi(G)$. 
\end{theorem} 

\begin{proof} The morphism $\phi$ gives rise to a symmetric tensor functor $\Phi: \Rep L\to \Rep G$. Consider the functor 
$\Phi':=\bold S\circ \Phi$. Even though $\bold S$ may not be exact on any side, the functor $\Phi'$ is exact since the category 
$\Rep L$ is semisimple (as $L$ is reductive). Thus, $\Phi': \Rep L\to \Rep \overline{G}$ is a symmetric tensor functor. Hence, by Tannakian formalism (\cite{DM}) it comes from a homomorphism $\phi': \overline{G}\to L$ defined uniquely up to conjugation in $L$. Moreover, consider the functor $\bold S^*\circ \Phi'=\bold S^*\circ \bold S\circ \Phi$. 
This functor is exact since its source is a semisimple category, so it is a symmetric tensor functor, and it follows from Proposition 13.7.1 of \cite{AK} that it is naturally isomorphic to $\Phi$ as a tensor functor. This means that the homomorphisms $\phi$ and $\phi'\circ \psi_G$ are conjugate under $L$: $\phi(g)=\ell \phi'(\psi_G(g))\ell^{-1}$ for some $\ell\in L$ and all $g\in G$. Hence, 
$\phi(g)=\widetilde{\phi}(\psi_G(g))$ for all $g\in G$, where $\widetilde{\phi}(a):=\ell \phi'(a)\ell^{-1}$. 

Finally, let us show that the homomorphism $\widetilde{\phi}$ in the theorem is determined uniquely up to conjugation in $L$ 
(automatically by elements commuting with $\phi(G)$). To this end, let $\widetilde\Phi: \Rep L\to \Rep G$ 
be the functor defined by $\widetilde{\phi}$. Then $\bold S^*\circ \widetilde{\Phi}=\Phi$, hence, postcomposing with $\bold S$, we get $\widetilde \Phi=\bold S\circ \Phi$. Thus, $\widetilde \Phi$ is uniquely determined and hence $\widetilde{\phi}$ is determined up to conjugation, as desired.    
\end{proof} 

\begin{remark} A geometric proof of the existence and properties of the reductive envelope is given in \cite{S}. 
\end{remark} 

\begin{example} Consider the special case $G=\Bbb G_a$. In this case the indecomposable representations of $G$ are unipotent Jordan blocks $J_n$ of sizes $n=1,2,3...$, so it is easy to see that $\overline{\Rep G}\cong \Rep SL(2)$ (as the Grothendieck ring of $\overline{\Rep G}$ 
coincides with that of $\Rep SL(2)$, and the dimensions of nonzero objects of $\overline{\Rep G}$ are positive). 
So in this case the existence of $\overline{G}$ is easy (namely, $\overline{G}=SL(2)$), and the existence of the splitting $\bold S^*$ is also straightforward (namely, $\bold S^*$ is induced by 
the standard inclusion $\psi_G: \Bbb G_a\hookrightarrow SL(2)$ 
as upper triangular matrices with ones on the diagonal). Thus, Theorem \ref{ak} in this case tells us that any homomorphism 
$\phi: \Bbb G_a\to L$ for a reductive proalgebraic group $L$ uniquely (up to conjugacy) factors through a homomorphism 
$\widetilde{\phi}: SL(2)\to L$. As pointed out in \cite{AK, S}, this implies the celebrated {\it Jacobson-Morozov Lemma}: 

\begin{proposition} Let $L$ be a reductive algebraic group over $\k$, and $u\in L$ a unipotent element. 
Then there exists a homomorphism $\theta: SL(2)\to L$ such that $\theta\left(\begin{matrix} 1&1\\ 0&1\end{matrix}\right)=u$. 
Moreover, $\theta$ is unique up to conjugation by the centralizer $Z_u$ of $u$.   
\end{proposition} 

\begin{proof} Let $G$ be the 1-parameter unipotent subgroup of $L$ generated by $u$, and $\phi: G\to L$ be the corresponding embedding. Identify $G$ with $\Bbb G_a$ by sending $u$ to $1$. Then it remains to apply Theorem \ref{ak} and set $\theta=\widetilde\phi$.  
\end{proof} 
\end{example} 

Note that when $G$ is an algebraic group then $\overline{G}$ is typically only a proalgebraic group (of infinite type), which can be very large. In fact, this is already so when $G=\Bbb G_a^2$, since the problem of classifying pairs of commuting matrices is well known to be wild; i.e., the case $G=\Bbb G_a$ (leading to the Jacobson-Morozov Lemma) is a rare exception. 
In other words, the whole category $\Rep \overline{G}$ is typically unmanageable. However, it makes sense to consider 
tensor subcategories of this category generated by a single object, which are more manageable. Namely, we have the following corollary.

\begin{corollary}\label{parenv} Let $G$ be an affine algebraic group over $\k$, and $V\in \Rep G$ a faithful representation of $G$ (so that $G\hookrightarrow GL(V)$). Then there exists a reductive algebraic group $G_V\subset GL(V)$ containing $G$ (a quotient of $\overline{G}$) such that the subcategory $\C_V$ of $\overline{\Rep G}\cong \Rep \overline{G}$ tensor generated by $\overline{V}$ is naturally equivalent to $\Rep G_V$.    
\end{corollary}    

\begin{proof} Let $F_V: \C_{\overline{V}}\to \Vec$ be the restriction of the fiber functor of $\Rep\overline{G}$ 
to $\C_{\overline{V}}$. Let $G_V:=\Aut(F_V)$. Then $G_V\subset GL(V)$ is a reductive subgroup such that 
$\Rep G_V=\C_{\overline{V}}$. Moreover, $G_V$ is a quotient of $\overline{G}$, hence we have a natural homomorphism 
$G\to G_V$, which is obviously injective, as desired. 
\end{proof} 

\begin{definition}\label{aksec} We will call $G_V$ the {\it reductive envelope of $G$ inside $GL(V)$}. 
\end{definition} 

\begin{remark} 1. Let $\C=\Rep_\k\Bbb Z/p$ , where ${\rm char}(\k)=p\ge 5$. Then a tensor functor $\bold S^*: \overline{\C}\to \C$ 
does not exist, since $\overline{\C}={\rm Ver}_p$ contains objects of non-integer Frobenius-Perron dimension. 
Also, if $\C=\Rep GL(n|1)$ over $\k$ of characteristic zero then a symmetric functor $\bold S^*$ as is Theorem \ref{splitting}  does not exist, either. 
Indeed, if $V$ is the vector representation of $GL(n|1)$ then $\wedge^{n-1}\bold S(V)=0$ (cf. Example \ref{ex1}(3)), while $\wedge^{n-1}V\ne 0$
(it is a negligible but nonzero object in $\C$). In fact, it is clear that a splitting functor $\bold S^*$ with the properties stated in Theorem \ref{splitting} cannot exist if $\C$ has indecomposable objects of dimension $0$. 

2. Note that the existence of the group $\overline{G}$ such that $\Rep \overline{G}\cong \overline{\Rep G}$ follows from Deligne's theorem (\cite{D1}, Theorem 7.1), since $\overline{\Rep G}$ is a symmetric tensor category over $\k$ in which nonzero objects have positive integer dimensions. This is, in fact, used in the proof of Theorem \ref{splitting} in \cite{AK}. 

Moreover, using a more general version of Deligne's theorem 
for supergroups, \cite{D2}, one can see that if $G$ is an affine proalgebraic supergroup over $\k$ of characteristic zero and $z\in G$ an element of 
order $\le 2$ acting on $O(G)$ by parity, and $\Rep(G,z)$ is the category of representations of $G$ on superspaces
on which $z$ acts by parity, then $\overline{\Rep(G,z)}=\Rep(\overline{G},\overline{z})$ for some 
reductive proalgebaric supergroup $\overline{G}$, i.e., one whose representation category is semisimple, see \cite{H}, Theorem 2.2. 
In particular, for each $V\in \Rep(G,z)$, $\overline{V}$ generates a category $\Rep(G_V,\overline{z})$, where $G_V$ is a reductive 
algebraic supergroup (a quotient of $\overline{G}$). This means that the connected component of the identity $G_V^0$ of $G_V$ 
is of the form $G_V'/C$, where $C$ is a finite central subgroup and $G_V'=G_V^+\times G_V^-$, where $G_V^+$ 
is a usual reductive group, and ${\rm Lie}G_V^-$ is a direct sum of Lie superalgebras of type $\mathfrak{osp}(1|2n)$, see \cite{W}.
\end{remark} 

In fact, as was explained to us by Thorsten Heidersdorf, the symmetric structure of the category is not essential in the Andr\'e-Kahn 
theorem  on the existence of splitting of the semisimplification functor. Namely, Theorem 12.1.1 and 13.2.1 
of \cite{AK} immediately imply the following theorem: 

\begin{theorem}\label{AKth} 
Let $\C$ be a Karoubian pivotal category as in Theorem \ref{karpiv}, such that the ideal ${\mathcal N}(\C)$ 
of negligible morphisms in $\C$ coincides with the nilpotent radical ${\rm rad}(\C)$ of $\C$ (where ${\rm rad}(\C)(X,Y)$ is the intersection 
of the radical of the algebra ${\rm End}(X\oplus Y)$ with $\Hom_\C(X,Y)$); in other words, $\C$ has no nonzero indecomposable objects 
of zero dimension. Then the semisimplification functor $\bold S: \C\to \overline{\C}$ admits a monoidal 
splitting $\bold S^*:\overline{\C}\to \C$. 
\end{theorem} 

\begin{corollary}\label{AK2} Let $\C$ be a Karoubian pivotal category as in Theorem \ref{karpiv}.
Then following conditions are equivalent: 

(i) Any indecomposable direct summand in a tensor product of two indecomposable objects
of nonzero dimension also has nonzero dimension; in other words, indecomposables of nonzero dimension 
span a full monoidal subcategory of $\C$.

(ii) The semisimplification functor $\bold S: \C\to \overline{\C}$ admits a monoidal 
splitting $\bold S^*:\overline{\C}\to \C$.
\end{corollary} 

\begin{proof}
It follows from Theorem \ref{AKth} that (i) implies (ii). To prove that (ii) implies (i), assume that (i) fails, and let 
$X,Y$ be simple objects of nonzero dimension such that $X\otimes Y\cong Z\oplus T$ where $Z\ne 0$ is indecomposable and has dimension zero. Then $\bold S(X)\otimes \bold S(Y)\cong \bold S(T)$.
So if there is a monoidal splitting $\bold S^*$, then applying $\bold S^*$ to the last equality, we get $X\otimes Y\cong T$. 
Thus, $T\cong Z\oplus T$, which contradicts the Krull-Schmidt theorem. 
\end{proof} 

In particular, this implies existence of reductive envelopes for quantum groups. Namely, let $H$ be a Hopf algebra 
over a field $\k$ in which the squared antipode is given by conjugation by a character $\chi\in H^*$ (thereby defining a pivotal structure on ${\rm Corep}(H)$), 
with no indecomposable finite dimensional comodules $M\ne 0$ of zero dimension. 
Then the category ${\rm Corep}(H)$ satisfies the assumptions of Theorem \ref{AKth}. 
Therefore, we obtain 

\begin{proposition}\label{AK1} 
There exists a unique universal cosemisimple Hopf algebra $\overline{H}$ with a surjective homomorphism 
$s: \overline{H}\to H$ (in other words, any Hopf algebra homomorphism 
$H'\to H$ from a cosemisimple Hopf algebra $H'$ uniquely factors through $s$). Namely, ${\rm Corep}(\overline{H})=\overline{{\rm Corep}(H)}$ \end{proposition} 

The proof is analogous to the case when $H$ is commutative over $\Bbb C$ (i.e., the case of proalgebraic groups). 
Heuristically writing $G={\rm Spec}(H)$ and $\overline{G}={\rm Spec}(\overline{H})$, we may say that the reductive quantum group
$\overline{G}$ is the reductive envelope of the quantum group $G$. 

\begin{example}  Let $q$ be a transcendental number, or, more generally, a complex number which is not a root of any polynomial with positive integer coefficients, say a positive real number or its image under an automorphism of $\Bbb C$. Let $B_q$ 
be the quantum Borel subgroup of the quantum group $G_q$ attached to a simple complex algebraic group $G$. 
The category $\Rep(B_q)={\rm Corep}(O(B_q))$ has a pivotal structure given by the element $q^{2\rho}$ of the corresponding quantum enveloping algebra. Therefore, the dimension (both left and right) of any nonzero representation of $B_q$ is a Laurent polynomial in $q$ with positive integer coefficients. Thus, this dimension is nonzero. Hence Proposition \ref{AK1} applies, and $H:=O(B_q)$ is the quotient of some cosemisimple Hopf algebra $\overline{H}:=O(\overline{B_q})$ for some reductive quantum group $\overline{B}_q$, such that every representation of $B_q$ factors canonically through $\overline{B_q}$.  
\end{example}

\subsection{Compatibility of semisimplification with equivariantization}

Now let $\C$ be a tensor category and $L$ be a finite group acting on $\C$.
Let $\C^L$ be the $L$-equivariantization of $\C$ (\cite{EGNO}, Subsection 4.15). 
The following lemma is easy (\cite{EGNO}, Exercise 4.15.3). 

\begin{lemma}\label{equivar} If 
$$
1\to N\to G\to L\to 1
$$
is an exact sequence of groups then $L$ acts naturally on $\Rep_{\bold k} N$, and $(\Rep_{\bold k} N)^L\cong \Rep_{\bold k} G$. 
\end{lemma}

Clearly, any action of $L$ on $\C$ descends to its action on the semisimplification $\overline{\C}$. 

\begin{proposition}\label{commu} If $|L|\ne 0$ in $\k$ and $L$ preserves the spherical structure of $\C$ then $L$-eqiuvariantization commutes with semisimplification. In other words, we have a natural equivalence of tensor categories $\overline{\C^L}\cong \overline{\C}^L$. 
\end{proposition}

\begin{proof} We have a natural forgetful functor $F: \C^L\to \C$. We claim that if $X,Y\in \C^L$ and 
$f: X\to Y$ is negligible, then $F(f)$ is negligible. Indeed, recall that $\Hom(F(X),F(Y))$ carries a natural action of $L$, and that $F$ defines an isomorphism $\Hom(X,Y)\cong \Hom(F(X),F(Y))^L$. Now let $h\in \Hom(F(Y),F(X))$, and let us show that $\Tr(F(f)\circ h)=0$. Let $\overline{h}=|L|^{-1}\sum_{\gamma\in L}\gamma(h)$. Then $\overline{h}=F(g)$ for a unique $g\in \Hom(X,Y)$. Thus, since $F(f)$ commutes with $L$ and the action of $L$ preserves traces, we have 
$$
\Tr(F(f)\circ h)=\Tr(F(f)\circ \overline{h})=\Tr(F(f)\circ F(g))=0,
$$
as desired. Thus, the functor $F$ descends to a tensor functor $\overline{F}: \overline{\C^L}\to \overline{\C}$. 
Moreover, for any $T\in \overline{\C^L}$ the object $\overline{F}(T)$ has a natural structure of an $L$-equivariant object (coming from that of $T$), 
so the functor $\overline{F}$ factors naturally through a tensor functor $E:  \overline{\C^L}\to \overline{\C}^L$. 

Suppose $T\in \overline{\C^L}$ is simple. Then $T=\overline{X}$, where $X\in \C^L$ is an indecomposable object of nonzero dimension. 
Thus $X=\Ind_{L_Z}^L(\rho\otimes Z)$, where $Z$ is an indecomposable object of $\C$ of nonzero dimension, $L_Z$ is the stabilizer 
of $Z$ in $L$, and $\rho$ is an irreducible representation of $L_Z$ over $\k$. Then 
$E(T)=\Ind_{L_Z}^L(\rho\otimes \overline{Z})$. Thus, $E(T)$ is simple (since so is $\overline{Z}$, and $L_{\overline{Z}}=L_Z$). 

It remains to show that $E$ is essentially surjective, i.e. every simple object of $\overline{\C}^L$ 
is of the form $E(T)$. To this end, note that every simple object of $\overline{\C}^L$ has the form 
$W=\Ind_{L_V}^L(\rho\otimes V)$, where $V=\overline{X}$ is a simple object of $\overline{\C}$ and $\rho$ is an irreducible representation of 
$L_V$. Since $|L_V|\ne 0$ in $\k$, we have $\dim \rho\ne 0$ in $\k$. Hence, $W=E(T)$, where $T=\Ind_{L_V}^L(\rho\otimes X)$
is a simple object of $\overline{\C^L}$ (as $L_V=L_X$). The proposition is proved. 
\end{proof} 

\begin{corollary}\label{commu2} In the setup of Lemma \ref{equivar} assume that $|L|\ne 0$ in $\k$. Then $\overline{\Rep_\k G}\cong \overline{\Rep_\k N}^L$.
\end{corollary} 
 
 \begin{proof} This follows from Proposition \ref{commu} and Lemma \ref{equivar}. 
 \end{proof}
 
\begin{remark} Similarly to Theorem \ref{karpiv}, Proposition \ref{commu} and its proof generalizes to Karoubian pivotal categories 
satisfying the assumptions of Theorem \ref{karpiv}. 
\end{remark} 

\subsection{Compatibility of negligible morphisms with surjective tensor functors}

Let $\C,\D$ be finite spherical tensor categories (\cite{EGNO}, Section 6), and $F: \C\to \D$ a surjective tensor functor (\cite{EGNO}, Subsection 6.3). Let $I: \D\to \C$ be the right adjoint of $F$. Note that $I$ is exact since $F$ maps projectives to projectives, 
(\cite{EGNO}, Theorem 6.1.16). 

\begin{definition} The {\it index} of $F$ is $d:=\dim I(\bold 1)$. 
\end{definition} 

\begin{definition} Let us say that $I$ is {\it dimension-scaling} if $\dim I(V)=d\dim V$ for all $V\in \D$. 
\end{definition} 

\begin{proposition}\label{negli} If $F$ has a nonzero index and $I$ is dimension-scaling then 

(i)  $\dim F(Y)=\dim Y$ for all $Y\in \C$; 

(ii) for any negligible morphism $f$ in $\C$, 
the morphism $F(f)$ is negligible in $\D$. 
\end{proposition}  

\begin{proof} We have a functorial isomorphism $\varepsilon_Y: I(F(Y))\to I(\bold 1)\otimes Y$. Indeed, 
$$
\Hom(X,I(F(Y)))=\Hom(F(X),F(Y))=\Hom(F(X)\otimes F(Y)^*,\bold 1)=
$$
$$
\Hom(F(X\otimes Y^*),\bold 1)=
\Hom(X\otimes Y^*,I(\bold 1))=
\Hom(X,I(\bold 1)\otimes Y). 
$$
Since $I$ is dimension-scaling, we have 
$$
d\dim F(Y)=\dim I(F(Y))=\dim (I(\bold 1)\otimes Y)=d \dim Y.
$$
Since $d\ne 0$, this implies (i). 

Now let us prove (ii). For this, note that if $f: X\to Y$ is a morphism in $\C$ then $\varepsilon_Y\circ I(F(f))\circ \varepsilon_X^{-1}={\rm Id}_{I(\bold 1)}\otimes f$. 
Hence the morphism $I(F(f))$ is negligible. 

\begin{lemma}\label{scaletrace} If $h: V\to V$ is a morphism in $\D$ then one has $\Tr (I(h))=d\Tr (h)$. 
\end{lemma} 

\begin{proof} By decomposing $V$ into generalized eigenobjects of $h$, we may assume that $h$ has a single eigenvalue $\lambda$. 
Then $h=\lambda{\rm Id}+h_0$, where $h_0$ is nilpotent, so $I(h)=\lambda {\rm Id}+I(h_0)$. Since $I(h_0)$ is nilpotent, the desired statement reduces to the identity $\dim I(V)=d\dim V$ for all $V\in \D$, which holds since $I$ is dimension-scaling. 
\end{proof} 

Now let $g: F(Y)\to F(X)$ be a morphism. Then by Lemma \ref{scaletrace}, 
$$
d\Tr(F(f)\circ g)=\Tr(I(F(f)\circ g))=\Tr(I(F(f))\circ I(g)).
$$
But this is zero, since $I(F(f))$ is negligible. Since $d\ne 0$, this implies that $\Tr(F(f)\circ g)=0$, i.e., $F(f)$ is negligible, yielding (ii). 
\end{proof} 

Proposition \ref{negli} immediately implies 

\begin{corollary}\label{desc} 
If $F$ has a nonzero index and $I$ is dimension-scaling then 
$F$ descends to a tensor functor $\overline{F}: \overline{\C}\to \overline{\D}$. 
\end{corollary}

Now let $H$ be an involutive finite dimensional Hopf algebra over any algebraically closed field $\k$ (i.e., $S^2=\Id$, where $S$ is the antipode of $H$), and $K$ be a Hopf subalgebra in $H$; for example, $H$ is cocommutative (e.g., a group algebra). Then $\C=\Rep H$ and $\D=\Rep K$ are finite spherical tensor categories, where dimensions are the usual dimensions (projected to $\k$). Restriction from $H$ to $K$ defines a surjective tensor functor $F: \C\to \D$. Let $I: \Rep(K)\to \Rep(H)$ be the right adjoint to this functor, i.e., the induction functor, 
$I(V)=\Hom_K(H,V)$. 

Recall that by the Nichols-Zoeller theorem \cite{NZ}, $H$ is a free $K$-module, of some rank $d$.

\begin{corollary}\label{invo} Assume that $d\ne 0$ in $\k$. Then any negligible morphism $f: X\to Y$ of $H$-modules is also negligible as a morphism of $K$-modules. Thus, $F$ defines a tensor functor: $\overline{\Rep H}\to \overline{\Rep K}$.   
\end{corollary}  

\begin{proof} 
We have $\dim I(V)=d\dim V$, i.e., $I$ is dimension-scaling. Thus, the result follows from Proposition \ref{negli}.
\end{proof} 

\section{Semisimplification of representation categories of finite groups in characteristic $p$.} 

\subsection{The result} 
Let ${\rm char}(\k)=p>0$. Let $G$ be a finite group, and $P$ be a Sylow $p$-subgroup of $G$. Let $N_G(P)$ be the normalizer of $P$ in $G$. Since $[G:N_G(P)]\ne 0$ in $\k$, Corollary \ref{invo} implies 

\begin{proposition} \label{neglres} Let $f: X\to Y$ be a negligible morphism of $G$-modules. Then $f$ is negligible as a morphism 
of $N_G(P)$-modules. Thus, the restriction functor $F: \Rep_\k G\to \Rep_\k N_G(P)$ descends to a tensor functor 
between semisimplifications $\overline{F}: \overline{\Rep_\k G}\to \overline{\Rep_\k N_G(P)}$.
\end{proposition}  

Our main result in this section is the following theorem.

\begin{theorem}\label{equiv} The functor $\overline{F}$ in Proposition \ref{neglres} is an equivalence of tensor categories. 
\end{theorem} 

Theorem \ref{equiv} is proved in the next subsection.

Let $L=N_G(P)/P$. 

\begin{corollary} One has $\overline{\Rep_\k G}\cong (\overline{\Rep_\k P})^L$. 
\end{corollary} 

\begin{proof} This follows from Theorem \ref{equiv} and Corollary \ref{commu2}, since $|L|\ne 0$ in $\k$ (as $P\subset G$ is a $p$-Sylow subgroup). 
\end{proof} 

\subsection{Proof of Theorem \ref{equiv}}

To prove Theorem \ref{equiv}, we will use the theory of vertices of modular representations and the Green correspondence (see e.g. \cite{A}, Chapter III), which we will now recall. Let $M$ be a finite dimensional representation of a finite group $G$ over a field $\k$ of characteristic $p$. Let $H$ be a subgroup of $G$. 

\begin{definition} We say that $M$ is relatively $H$-projective if $M$ is a direct summand in ${\rm Ind}_H^GV$ for some finite dimensional $H$-module $V$. 
\end{definition} 

\begin{proposition} (see e.g. \cite{A}, Section 9)\label{conju} For each indecomposable $G$-module $M$, the minimal subgroups $H\subset G$ such that $M$ is relatively $H$-projective are conjugate, and they are $p$-groups. 
\end{proposition}  

\begin{definition} The minimal subgroup $H\subset G$ such that $M$ is relatively $H$-projective (well-defined up to conjugation thanks to Proposition \ref{conju}) is called the {\it vertex} of $M$. 
\end{definition} 

\begin{proposition} \label{dimnonzero} (\cite{G}, Theorem 9) If $\dim M\ne 0$ in $\k$ then the vertex of $M$ is the Sylow $p$-subgroup $P\subset G$. 
\end{proposition} 

\begin{proof} The result is well known, but we give a proof for reader's convenience. 
Let $H$ be the vertex of $M$, so $M$ is a direct summand of ${\rm Ind}_H^GV$ for some 
$H-$module $V$. For the sake of contradiction assume that $H$ is not conjugate
to $P$. We will prove: 

(a) any direct summand of ${\rm Ind}_H^GV$ has dimension zero.

This is a contradiction with our assumption on $H$, since $M$ is one of such direct summands. 
We deduce (a) from the following stronger statement:

(b) any direct summand of ${\rm Res}^G_P{\rm Ind}_H^GV$ has dimension zero.

To prove (b), recall that by the Mackey formula (see \cite[III.8, Lemma 7]{A}):
$$
{\rm Res}^G_P{\rm Ind}_H^GV=\bigoplus_{s\in P\setminus G/H}{\rm Ind}_{P\cap sHs^{-1}}^P
{\rm Res}^{sHs^{-1}}_{P\cap sHs^{-1}}s(V).
$$
By the assumption $P\cap sHs^{-1}$ is strictly contained in $P$ for any $s$. Since $P$
is $p-$group, $P\cap sHs^{-1}$ is a subnormal subgroup of $P$.
Thus by
Green's indecomposability Theorem (see \cite[III.8, Theorem 8]{A}) the functor
${\rm Ind}_{P\cap sHs^{-1}}^P$ sends indecomposable modules to indecomposable ones.
In particular, any direct summand of 
${\rm Ind}_{P\cap sHs^{-1}}^P{\rm Res}^{sHs^{-1}}_{P\cap sHs^{-1}}s(V)$ is induced from 
$P\cap sHs^{-1}$ and has 
dimension divisible by the index $[P:P\cap sHs^{-1}]$, hence vanishes in $\k$, see \cite[III.8, Lemma 4]{A}. 
The result follows.
\end{proof} 

\begin{theorem}\cite{A}\label{gk} (Green's correspondence) For each $p$-subgroup $H\subset N_G(P)$, there is a bijection 
between indecomposable representations of $G$ with vertex $H$ and indecomposable representations of $N_G(P)$ 
with vertex $H$, given by $X\mapsto X^\circ$ for $X\in \Rep_\k G$, such that 
$X|_{N_G(P)}=X^\circ\oplus N$, where $N$ is a direct sum 
of indecomposable $N_G(P)$-modules with vertices other than $H$. 
\end{theorem}

We can now prove Theorem \ref{equiv}. Let $T\in \overline{\Rep G}$ be a simple object. 
Then $T=\overline{X}$, where $X\in \Rep G$ is an indecomposable module of nonzero dimension. 
Hence, by Proposition \ref{dimnonzero}, the vertex of $X$ is $P$. Hence, 
by Theorem \ref{gk}, $X|_{N_G(P)}=X^\circ\oplus N$, where $N$ is a direct sum 
of indecomposable $N_G(P)$-modules whose vertices are different from $P$. 
Then by Proposition \ref{dimnonzero}, the dimension of each of these indecomposable modules is zero, hence 
$N$ is negligible. This means that $\overline{F}(T)=\overline{X^\circ}$, which is a simple object of $\overline{\Rep N_G(P)}$. 
This shows that the functor $\overline{F}$ is injective. 

Now let $Z\in \overline{\Rep N_G(P)}$ be a simple object. Then $Z=\overline{Y}$ for some $Y\in \Rep N_G(P)$.
But $Y$ is a direct summand in $({\rm Ind}_{N_G(P)}^GY)|_{N_G(P)}$. Hence, 
$Z$ is a direct summand in $\overline{F}(\overline{{\rm Ind}_{N_G(P)}^GY})$, proving that $\overline{F}$ is surjective. 

Thus, $\overline{F}$ is an equivalence, as claimed. 

\subsection{The case of Sylow subgroup of prime order} 

Let us now consider the simplest nontrivial special case of Theorem \ref{equiv}, 
when the $p$-Sylow subgroup of $G$ has order $p$. 

\begin{corollary}\label{cyc} If $P=\Bbb Z/p$ then $\overline{\Rep_\k G}=({\rm Ver}_p)^L$,
where ${\rm Ver}_p$ is the Verlinde category (see Example \ref{ex1}(2)).
\end{corollary} 

Note that if $p=2$ then ${\rm Ver}_p=\Vec$. Thus, if $P=\Bbb Z/2$ then Corollary 
\ref{cyc} says that $\overline{\Rep_\k G}=\Rep_\k L$. 

So let us consider the case $p>2$ and compute the category $({\rm Ver}_p)^L$ more explicitly. Note that ${\rm Ver}_p={\rm Ver}_p^+\boxtimes {\rm Supervec}$, and ${\rm Ver}_p$ has no nontrivial symmetric tensor autoequivalences (\cite{O}), while ${\rm Ver}_p^+$ has no nontrivial tensor automorphisms of the identity functor. Thus, from Corollary \ref{cyc} we get $$\overline{\Rep_\k G}={\rm Ver}_p^+\boxtimes {\rm Supervec}^L.$$ 

The group of tensor automorphisms of the identity functor of ${\rm Supervec}$ is $\Bbb Z/2$. 
Hence, actions of $L$ on ${\rm Supervec}$ correspond to elements $H^2(L,\Bbb Z/2)$. Let $c\in H^2(L,\Bbb Z/2)$ be the element corresponding to the action as above, and 
let us compute $c$. Since the action of $L$ on $\Bbb Z/p$ factors through an action of $\Bbb Z/(p-1)$, the element $c$ is pulled back from a canonical element 
$\bar c\in H^2(\Bbb Z/(p-1),\Bbb Z/2)=\Bbb Z/2$.

\begin{proposition}\label{nontri} The element $\bar c$ is nontrivial. 
\end{proposition}  

\begin{proof} 
It suffices to show that the pullback of $\bar c$ to $\Bbb Z/2\subset \Bbb Z/(p-1)$ is nontrivial. For this purpose, it suffices to consider the semisimplification 
of $\Rep_\k D_p$, where $D_p:=\Bbb Z/2\ltimes \Bbb Z/p$ is the dihedral group. In $\overline{\Rep_\k D_p}$ we have an invertible
object $X$ of vector space dimension $p-1$, which has composition series $\k_+,\k_-,...,\k_+,\k_-$, 
where $\k_+$ is the trivial representation of $\Bbb Z/2$ and $\k_-$ is the sign representation, and it suffices to show  
that $X$ has order $>2$. But we have $X=X^*\otimes \k_-$. Thus, $X$ cannot have order $2$, as desired.      
\end{proof} 

Let $\widetilde L$ be the central extension of $L$ by $\Bbb Z/2$ defined by the cocycle $c$, and let $z$ be the generator of the central subgroup $\Bbb Z/2\subset \widetilde L$.  

\begin{corollary}\label{corro} If $p$ is odd and $P=\Bbb Z/p$ then 
$$\overline{\Rep_\k G}\cong {\rm Ver}_p^+\boxtimes \Rep_\k(\widetilde L,z),$$ where $\Rep_\k(\widetilde L,z)$ is the category of representations 
of $\widetilde L$ on supervector spaces, so that $z$ acts by the parity operator.
\end{corollary} 

\subsection{The case of the symmetric group $S_{p+n}$, where $n<p$} \label{symgroup}

If $p=2$ then we have $\overline{\Rep_\k S_2}=\overline{\Rep_\k S_3}=\Vec_\k$. 
So consider the case $p>2$. Let $G=S_{p+n}$, where $0\le n<p$. Then $P=\Bbb Z/p$, and $N_G(P)=S_n\times \Bbb Z/(p-1)\ltimes \Bbb Z/p$. Thus, 
by Corollary \ref{corro}, 
$$
\overline{\Rep_\k S_{p+n}}\cong \Rep_\k S_n\boxtimes {\rm Ver}_p^+\boxtimes \Rep_\k (\Bbb Z/2(p-1),z),
$$
where $z$ is the element of order $2$ in $\Bbb Z/2(p-1)$. 
In particular, for $n\ge 2$ the group of invertible objects of this category is \linebreak $\Bbb Z/2\times \Bbb Z/2(p-1)$. 

In particular, we obtain the following proposition.

\begin{proposition} \label{eqi} If $n<p$ then the restriction functor 
$${\rm Res}: \Rep_\k S_{n+p}\to \Rep_\k (S_n\times S_p)$$ 
induces an equivalence $\overline{\Rep_\k S_{n+p}}\to \overline{\Rep_\k (S_n\times S_p)}$.
\end{proposition} 

\begin{proof}
The functor ${\rm Res}$ descends to a tensor functor $\overline{\Rep_\k S_{n+p}}\to \overline{\Rep_\k (S_n\times S_p)}$ by Corollary \ref{desc}, 
and this tensor functor is an equivalence since the inclusion $S_n\times S_p\hookrightarrow S_{n+p}$ induces an isomorphism of the normalizers of the Sylow $p$-subgroups. 
\end{proof} 

Let us now describe the functor $\bold S$ more explicitly, in the special case $n=0$, i.e., $\C=\Rep_\k S_p$, where $p>2$. 
It is well known that in this case we have a unique non-semisimple block $\mathcal B$ of defect $1$, namely, the block of the trivial representation. 
The blocks of defect zero consist of objects of dimension $0$, so they are killed by $\bold S$. 
So let us first consider the images under $\bold S$ of the simple objects of $\mathcal B$. These objects have the form $\wedge^iV_{p-2}$, $i=0,...,p-2$, where $V_{p-2}$ is the $p-2$-dimensional irreducible representation of $S_p$ which is the middle composition factor in the permutation representation. To compute the image $\bold S(V_{p-2})$  of $V_{p-2}$, denote by $L_i$ $i=1,3,5,...,p-2$ the simple objects of ${\rm Ver}_p^+$ (so that $L_1=\bold 1$), and by $\chi$ the generator of $\Rep_\k(\Bbb Z/2(p-1))$. The object $\bold S(V_{p-2})$ has to be simple and has dimension $-2$, so it has the form $L_{p-2}\otimes \chi^m$, where $m$ is even as $\dim \chi=-1$. Moreover, $S^2V_{p-2}$ contains $\bold 1$ as a direct summand, which implies that $m=0$ or $m=p-1$. Finally, $\wedge^{p-2}V_{p-2}={\rm sign}$ is the sign representation of $S_p$, so 
$\wedge^{p-2}(L_{p-2}\otimes \chi^m)=\chi^{m(p-2)}$ is nontrivial, which implies that $m=p-1$. Thus, 
$$
\bold S(V_{p-2})=L_{p-2}\otimes \chi^{p-1}.
$$
This means that 
$$
\bold S(\wedge^i V_{p-2})=L_{p-1-i}\otimes \chi^{p-1}
$$
for odd $i\le p-2$, and 
$$
\bold S(\wedge^i V_{p-2})=L_{i+1}
$$
for even $i\le p-2$. 

Now consider the representation $V_{p-1}$ of $S_p$ on the space of functions on $[1,p]$ modulo constants. 
Then $\bold S(V_{p-1})$ has dimension $-1$, so it is of the form $\chi^m$ for some odd $m$. Moreover, it is well known that 
$\wedge^iV_{p-1}$ is indecomposable for $i\le p-1$. Since it is not invertible for $0<i<p-1$, we see that $\chi^{mi}\ne \bold 1,\chi^{p-1}$ for any 
$0<i<p-1$. Also $\chi^{m(p-1)}=\chi^{p-1}$. This implies that the order of $\chi^m$ is $2(p-1)$, so we may assume that $m=1$ by making a suitable choice of $\chi$. Thus, for a suitable choice of $\chi$ we have 
$$
\bold S( V_{p-1})=\chi.
$$
The suitable choice of $\chi$ is well defined only up to the change $\chi\to \chi^p$, since the group $\Bbb Z/2(p-1)$ 
has an automorphism of order $2$ (sending $1$ to $p$) which acts trivially on $\Bbb Z/(p-1)=\Aut(\Bbb Z/p)$. Thus, the well-defined question is to determine $\chi^2$, which is a character of $\Aut(\Bbb Z/p)$, naturally identified with $\Bbb F_p^\times$. 
Then it is easy to show by a direct calculation that $\chi^2$ is the natural inclusion $\Bbb F_p^\times\hookrightarrow \k^\times$ coming from the inclusion of fields $\Bbb F_p\hookrightarrow \k$. 

Thus, we obtain

\begin{proposition}\label{generat} The category $\overline{\Rep_\k S_p}$ is generated by $\overline{V_{p-2}}$ and $\overline{V_{p-1}}$. In other words, 
the simple objects of $\overline{\Rep_\k S_p}$ have the form $\overline{V_{p-1}}^{\otimes m}\otimes \wedge^i \overline V_{p-2}$, where 
$0\le m\le p-2$ and $0\le i\le p-2$ (so the total number of simple objects is $(p-1)^2$). 
\end{proposition}

\subsection{Application: the semisimplification of the Deligne category $\underline{\Rep}^{\rm ab} S_n$}

Let $n$ be a nonnegative integer, and $\k$ be a field of characteristic zero. Let $\underline{\Rep} S_n$ denote the Karoubian Deligne category over $\k$ defined in \cite{D3} (its main property is that it can be interpolated to non-integer values of $n$ in $\k$). This category has a tensor ideal $I$ such that $\underline{\Rep} S_n/I=\Rep_\k S_n$. Moreover, it is known (see \cite{D3},\cite{CO}) that    $\underline{\Rep} S_n$ has an abelian envelope $\underline{\Rep}^{\rm ab} S_n$; in particular, the trace of any nilpotent endomorphism in $\underline{\Rep} S_n$ vanishes. 
Since $I$ consists of morphisms factoring through negligible objects (i.e., direct sums of indecomposable objects of dimension zero), and $\Rep_\k S_n$ is semisimple, we see 
that $I=\N(\C)$ is the full ideal of negligible morphisms (i.e., every negligible morphism factors through a negligible object), and 
the semisimplification $\overline{\underline{\Rep} S_n}$ coincides with $\Rep_\k S_n$.

The question of describing the semisimplification of the abelian envelope $\underline{\Rep}^{\rm ab} S_n$ is more interesting. The answer is given by the following theorem. 

\begin{theorem}\label{semdel} 
(i) The restriction functor 
$$
{\rm Res}: \underline{\Rep}^{\rm ab} S_n\to \Rep_\k S_n\boxtimes \underline{\Rep}^{\rm ab} S_0
$$ 
induces an equivalence between the semisimplifications of these categories. 

(ii) We have an equivalence of symmetric tensor categories $\overline{\underline{\Rep}^{\rm ab} S_n}\cong \Rep_\k S_n\boxtimes \Rep_\k(GL(1)\times SL(2),(-1,-1))$.
\end{theorem} 

\begin{proof} 
We will use the approach of \cite{Ha} to Deligne categories. Namely, let us take $\k=\Bbb C$. 
Then, according to \cite{Ha}, Theorem 1.1(b), we have 
$$
\underline{\Rep}^{\rm ab} S_n=\lim_{p\to \infty} \Rep_{\overline{\Bbb F}_p}S_{n+p}, 
$$
where $\lim$ denotes an appropriate ultrafilter limit (i.e., ultraproduct). More precisely, this means that $\underline{\Rep}^{\rm ab} S_n$ 
is the tensor subcategory in the appropriate ultrafilter limit tensor generated by the "permutation" object $P$ (the analog of the permutation representation). It is easy to see that the ultrafilter limit commutes with the semisimplification, so (i) follows from Proposition \ref{eqi}. 

By virtue of (i), it suffices to check (ii) for $n=0$. In this case, according to Subsection \ref{symgroup}, $\overline{\Rep_{\overline{\Bbb F}_p}S_{n+p}}=\overline{\Rep_{\overline{\Bbb F}_p}S_{p}}$ is generated by $\overline{V_{p-2}}$ and $\overline{V_{p-1}}$. In the ultrafilter limit, the sequences of representations $V_{p-2}$ and $V_{p-1}$ converge to the objects $V_{-2}$ and $V_{-1}$ of $\underline{\Rep}^{\rm ab} S_0$ (of dimensions $-2$ and $-1$, respectively), defined by the (non-split) exact sequences 
$$
0\to \bold 1\to P\to V_{-1}\to 0,\ 0\to V_{-2}\to V_{-1}\to \bold 1\to 0
$$ 
(in particular, $V_{-2}$ is simple). Thus, by Proposition \ref{generat}, the category $\overline{\underline{\Rep}^{\rm ab} S_n}$ is generated by $\overline{V_{-2}}$ and $\overline{V_{-1}}$. Moreover, since 
$$\overline{V_{p-2}}=L_{p-2}\otimes \chi^{p-1},$$ we find that $\overline{V_{-2}}$ generates a subcategory 
with Grothendieck ring of $\Rep_\k SL(2)$. Since $\dim V_{-2}=-2$, this is the category $\Rep_\k(SL(2),-1)$. 
Similarly, since $\overline{V_{p-1}}$ is invertible of order $2(p-1)$, we see that $\overline{V_{-1}}$ is invertible of infinite order, so 
since its dimension is $-1$, it generates $\Rep_\k(GL(1),-1)$. Thus, together these two objects generate the category $\Rep_\k S_n\boxtimes \Rep_\k(GL(1)\times SL(2),(-1,-1))$, as claimed. 
\end{proof} 

\section{Semisimplification of some non-symmetric categories}

Let ${\rm char}(\k)=0$, $q\in \k^\times$, and $H_q$ be the Hopf algebra generated by the grouplike element 
$g$ and element $E$ with defining relation $gEg^{-1}=qE$  and coproduct defined by $\Delta(E)=E\otimes g+1\otimes E$. Then $S(E)=-Eg^{-1}$, so $S^2(E)=gEg^{-1}=qE$. Let $\C_q\subset \Rep H_q$ be the 
category of finite dimensional representations of $H_q$ on which $g$ acts semisimply with eigenvalues being powers of $q$. This category has a pivotal structure defined by the element $g$.  

\subsection{Generic $q$} 
First assume that $q$ is not a root of unity. Then for any
$V\in \C_q$, $E|_V$ is nilpotent, since $E$ maps eigenvectors of $g$
with eigenvalue $\lambda$ to those with eigenvalue $\lambda q$.  Thus,
the indecomposable objects of $\C_q$ are $V_{m_1,m_2}$, where
$m_1\ge m_2$ are integers, namely, Jordan blocks for $E$ of size
$m_1-m_2+1$ containing a nonzero vector $v$ with $gv=q^{m_1}v, Ev=0$.
Then $\dim V_{m_1,m_2}=q^{m_2}+...+q^{m_1}$, which is never zero, so there is no nonzero negligible objects.
It is easy to see that the tensor product of $V_{m_1,m_2}$ obeys the
same fusion rules as representations of $GL_{{\bold q}}(2)$ with highest weights
$(m_1,m_2)$, where ${\bold q}^2=q$. From this we obtain

\begin{proposition}\label{quantgl2}
One has $\overline{\C_q}\cong \Rep GL_{{\bold q}}(2)$. 
\end{proposition} 

\begin{proof}
Let us construct a tensor functor $T: \Rep GL_{{\bold q}}(2)\to \C_q$ such that $\bold S\circ T$ 
is an equivalence $\Rep GL_{{\bold q}}(2)\to \overline{\C_q}$. For this purpose, consider the Hopf algebra $U_{{\bold q}}(\mathfrak{gl}_2)$ 
with generators $g_1,g_2,e,f$ such that $g_1,g_2$ are commuting grouplike elements and 
$$
g_1eg_1^{-1}={\bold q}e,\ g_1fg_1^{-1}={\bold q}^{-1}f,\ 
g_2eg_2^{-1}={\bold q}^{-1}e,\ g_2fg_2^{-1}={\bold q}f,\ 
$$
$$
[e,f]=\frac{g_1g_2^{-1}-g_2g_1^{-1}}{{\bold q}-{\bold q}^{-1}},
$$
$$
\Delta(e)=e\otimes g_1g_2^{-1}+1\otimes e,\ \Delta(f)=f\otimes 1+g_2g_1^{-1}\otimes f.
$$
Let us realize $\Rep GL_{{\bold q}}(2)$ as the category of 
finite dimensional representations 
of $U_{{\bold q}}(\mathfrak{gl}_2)$ on which $g_1,g_2$ act semisimply with eigenvalues being powers of ${\bold q}$. 
Let $J$ be the twist for $U_{{\bold q}}(\mathfrak{gl}_2)$ which acts on $v\otimes w$ by ${\bold q}^{-rs}$ when $g_1v={\bold q}^rv$ and $g_2w={\bold q}^sw$. 
Then the conjugated coproduct $\Delta_J(a):=J^{-1}\Delta(a)J$ of the element $e$ has the form  
$$
\Delta_J(e)=e\otimes g_1+g_1^{-1}\otimes e.
$$  
Thus, setting $\bar e:=g_1e$, we have 
$$
\Delta_J(\bar e)=\bar e\otimes g_1^2+1\otimes \bar e.
$$
We therefore have an inclusion of Hopf algebras $\psi: H_q\hookrightarrow U_{{\bold q}}(\mathfrak{gl}_2)^J$ given by $\psi(g)=g_1^2$, $\psi(E)=\bar e$,
which defines the desired tensor functor $T$.  
\end{proof} 
  
\subsection{Roots of unity} 
Now consider the case when $q$ is a root of unity of some order $n$, which is more interesting. For simplicity assume that $n\ge 3$ is odd, and let $\bold q$ be a root of unity of order $2n$ such that $\bold q^2=q$. 
In this case, by definition, $\C_q=\Rep H_q/(g^n-1)$ is the category of finite dimensional representations of the quotient Hopf algebra $H_q/(g^n-1)$. 
Note that the action of $E$ on objects of $\C_q$ no longer needs to be nilpotent. 
Namely, $E^n$ is a central element which can act on a simple module by an arbitrary scalar. 
However, if $E^n=\lambda\ne 0$ on some simple module $V$, then given an eigenvector $v\in V$ of $g$ with eigenvalue $\gamma$, the elements $v,Ev,...,E^{n-1}v$ are a basis of $V$, so $V$ has dimension 
$\gamma (1+q+q^2+...+q^{n-1})=0$. Thus, the action of $E$ on any non-negligible 
indecomposable module must be nilpotent. This shows that the non-negligible indecomposable 
modules are still $V_{m_1,m_2}$, but now $d:=m_1-m_2+1$ is not divisible by $n$, and also 
$m:=m_1$ is defined only up to a shift by $n$. We will denote this module by $V(m,d)$. 
Thus, the simple objects of $\overline{\C_q}$ are $\overline{V(m,d)}$, where
$0\le m\le n-1$ and $d\ge 1$, not divisible by $n$. Note that  $V(m,1)\otimes V(r,d)=V(r,d)\otimes V(m,1)=V(r+m,d)$ (with addition mod $n$), thus 
$\overline{V(m,1)}\otimes \overline{V(r,d)}=\overline{V(r,d)}\otimes \overline{V(m,1)}=\overline{V(m+r,d)}$. 

To compute the fusion rules in $\overline{\C_q}$, consider the Hopf subalgebra $K_q\subset H_q$ 
generated by $g$ and $E^n$ (this Hopf algebra is commutative and cocommutative, as $E^n$ is a primitive element). Let $\chi$ be the generating character of the cyclic group generated by $g$ such that $\chi(g)=q$. 
Then the Green ring of the category of finite dimensional representations of $K_q/(g^n-1)$ with nilpotent action of $E^n$ is 
$R[\Bbb Z/n]=R[\chi]/(\chi^n-1)$, where $R$ is the representation ring of $SL(2)$. 
Moreover, if $X\in \C_q$ is a negligible indecomposable module over $H_q/(g^n-1)$ then its restriction to $K_q/(g^n-1)$ lies in the ideal 
of $R[\Bbb Z/n]$ generated by $1+\chi+...+\chi^{n-1}$. 
Thus we have a natural homomorphism 
$$
\theta: {\rm Gr}(\overline{\C_q})\to R[\chi]/(1+\chi+...+\chi^{n-1}).
$$
Let us now compute $\theta(\overline{V(m,d)})$. First, it is clear that $\theta(\overline{V(m,1)})=\chi^m$. Also, 
for a simple object $X\in \overline{\C_q}$, let $\nu(X)\in \Bbb Z/2n$ be defined by $\nu(\overline{V(m,d)})=2m-d+1$. 
Then for any direct summand $Z$ in $X\otimes Y$ we have $\nu(Z)=\nu(X)+\nu(Y)$ (since the representations 
$V(m,d)$ extend to $GL_{{\bold q}}(2)$, where the order of ${\bold q}$ is $2n$, and ${\bold q}^{2m-d+1}={\bold q}^{m_1+m_2}$ is determined by the action of the central element $g_1g_2$). 
Thus, the subcategory $\C_q^0$ spanned by $V(m,d)$ with $2m-d+1=0$ modulo $2n$, is a tensor subcategory of $\C_q$. Moreover, it is easy to check that 
the restriction 
$$
\theta: {\rm Gr}(\overline{\C_q^0})\to R[\chi]/(1+\chi+...+\chi^{n-1})
$$
is injective. 

Now, the basis of ${\rm Gr}(\overline{\C_q^0})$ is formed by $\overline{V(m,2rn+2m+1)}$, $r\ge 0$. Consider first the case $r=0$, $0\le m\le \frac{n-3}{2}$. In this case, we get 
$$
\theta(\overline{V(m,2m+1)})=\chi^m+\chi^{m-1}+....+\chi^{-m}. 
$$
This means that the collection of $(n-1)/2$ objects $\overline{V(m,2m+1)}, 0\le m\le (n-3)/2$ 
span a tensor subcategory, whose Grothendieck ring is that of ${\rm Ver}_{\bold q}^+$, 
the even part of the category ${\rm Ver}_{\bold q}$ (the fusion category attached to $U_{\bold q}(\mathfrak{sl}_2)$). 

Now, let $W_i\in R$ be the $i+1$-dimensional irreducible representation of $SL(2)$. 
Then it is easy to see (by looking at bases of representations) that 
$$
\theta(\overline{V(0,2rn+1)})=W_{2r+1}-W_{2r},\ \theta(\overline{V(-1,2rn-1)})=W_{2r-1}-W_{2r}, \ r\ge 1.
$$
This means that the collection of objects $\overline{V(0,2rn+1)}$, $\overline{V(-1,2rn-1)}$, $r\ge 1$ spans a tensor subcategory with Grothendieck ring of $\Rep PGL(2)$, with $\overline{V(0,2rn+1)}\mapsto U_{4r+1}$, 
$\overline{V(-1,2rn-1)}\mapsto U_{4r-1}$, with $U_s$ denoting the irreducible representation of $PGL(2)$ of dimension $s$. Indeed, let us evaluate the characters of $W_i$ at the point $-x$. Then we have 
$$
\overline{V(0,2rn+1)}\mapsto x^{2r+1}+x^{2r}...+x^{-2r-1}, \overline{V(-1,2rn-1)}\mapsto x^{2r}+x^{2r-1}....+x^{-2r},
$$
which implies the statement. 

We also note that the object $V(n-1,n-1)$ is invertible and has order $2$. 

The analysis of the case when $n$ is even is similar, using Theorem \ref{fmps}. 

Thus we obtain 

\begin{theorem} \label{extprod}
The Grothendieck ring of $\overline{\C_q}$ is isomorphic to the Grothendieck ring of the category 
$$
\Vec_{\Bbb Z/n}\boxtimes {\rm Ver}_{\bold q}\boxtimes \Rep PGL(2).  
$$
\end{theorem}  

\begin{corollary} (i) The category spanned by $\overline{V(0,2rn+1)}$, $\overline{V(-1,2rn-1)}$
is a tensor category equivalent to $\Rep OSp(1|2)$. 

(ii) The category spanned by $\overline{V(m,2m+1)}, \overline{V(m,2m+1)}\otimes \overline{V(n-1,n-1)}$, $0\le m\le (n-3)/2$ is  a tensor category equivalent to ${\rm Ver}_{\bold q}$. 
\end{corollary} 

\begin{proof} Part (i) follows fromTheorem \ref{extprod} and Theorem \ref{mps} (ii) 
(since the generating object $V(-1,2n-1)$ corresponding to $U_3$ has dimension $-1$). 

Part (ii) follows from Theorem \ref{extprod} and Theorem \ref{fmps}, Remark \ref{anoth}(iii). 
\end{proof} 

Thus we expect that there is an equivalence of tensor categories
$$
\overline{\C_q}\cong \Vec_{\Bbb Z/n}\boxtimes {\rm Ver}_{\bold q}\boxtimes \Rep OSp(1|2).  
$$
Note that this does not immediately follow from Theorem \ref{extprod} since the external tensor product 
$\C \boxtimes \D$
might have nontrivial associators (for instance this is the case when both categories $\C$ and
$\D$ are pointed).

\section{Surjective symmetric tensor functors between Verlinde categories $\Ver_p(G)$}
 
Let $G$ be a simple algebraic group over $\Bbb Z$, $h=h(G)$ the Coxeter number of $G$, and $p\ge h$ a prime. Let $\k$ be an algebraically closed field of characteristic $p$. Let $\Ver_p(G)=\Ver_p(G,\k)$ be the associated Verlinde category of $G$, i.e., the semisimplification of the category ${\rm Tilt}(G)$ of tilting modules for $G$ over $\k$. For example, $\Ver_p(SL(2))=\Ver_p$. 

Similarly one defines $\Ver_p(G)$ when $G$ is connected reductive. In this case we should require that $p\ge h_i$ for all $i$, where $h_i$ are the Coxeter numbers 
of all simple constituents of $G$. Note that $\Ver_p(G)$ is a fusion category (i.e., finite) if and only if $G$ is semisimple.  

We would like to construct surjective symmetric tensor functors $\Ver_p(G)\to \Ver_p(K)$ for simple $G$. To this end, suppose that $\phi: K\hookrightarrow G$ 
is an embedding of reductive algebraic groups. In this case, we have the following proposition. 

\begin{proposition}\label{tiltres} Let $p$ be sufficiently large, and  
let $T$ be a tilting module for $G$. Then $T|_K$ is also a tilting module.  
\end{proposition}  

\begin{proof} The module $T$ occurs as a direct summand in $V^{\otimes m}$, where $V$ is the direct sum of the irreducible $G$-modules whose highest weights generate the cone of dominant weights for $G$. Hence $T|_K$ is a direct summand in $V^{\otimes m}|_K$. But $V|_K$ is a direct sum of simple $K$-modules with small highest weights (compared to $p$), which are therefore tilting. Thus, $T|_K$ is tilting.   
\end{proof} 

\begin{proposition}\label{neglii} Let $p$ be sufficiently large, and let $K$ contain a regular unipotent element of $G$ (equivalently,  a principal $SL(2)$-subgroup of $G$). Then for any negligible tilting module $T$ over $G$, the restriction $T|_K$ is negligible. 
\end{proposition} 

\begin{proof} 
Let $u\in K(\k)$  be a regular unipotent element of $G$, and $U\cong \Bbb Z/p$ be the subgroup generated by $u$.  Then by \cite{J}, E13, $T|_U$ is projective, hence negligible. This implies that $T|_K$ is negligible.  
\end{proof} 
 
\begin{corollary}\label{funsemisi} If $K$ contains a regular unipotent element of $G$ then for large enough $p$ we have a surjective tensor functor 
$F: \Ver_p(G)\to \Ver_p(K)$. 
\end{corollary}  

\begin{proof} By Proposition \ref{tiltres}, we 
have a monoidal functor $${\rm Res}: {\rm Tilt}(G)\to {\rm Tilt}(K),$$ and by Proposition \ref{neglii}, it maps negligible objects to negligible ones. 
Hence, this functor descends to a tensor functor between the semisimplifications $\overline{\rm Res}:  \overline{{\rm Tilt}(G)}\to \overline{{\rm Tilt}(K)}$.
This implies the required statement, since $\overline{{\rm Tilt}(G)}\cong \Ver_p(G)$ (and similarly for $K$), so we can take $F=\overline{\rm Res}$, and it is clear that this functor is surjective. 
\end{proof} 

Corollary \ref{funsemisi} raises a question of classification of pairs $K\subset G$, where $G$  is simple, $K$ is connected reductive, and $K$ contains a regular unipotent element of $G$. Let us call such a pair a {\it principal pair}. It is clear that it suffices to classify 
the corresponding pairs of Lie algebras (which we also call principal); namely, a principal pair of groups $K\subset G$ is determined by a principal pair of Lie algebras ${\mathfrak{k}}\subset \g$ and a central subgroup in $G$. The question of classification of principal pairs of Lie algebras  is solved by the following theorem. 

\begin{theorem} \cite{SS}\label{SSthm} The principal pairs of Lie algebras $\mathfrak{k}\subset \mathfrak{g}$  (with a proper inclusion) are given by the following list: 

(1) $\mathfrak{sp}(2n)\subset \mathfrak{sl}(2n)$, $n\ge 2$; 

(2) $\mathfrak{so}(2n+1)\subset \mathfrak{sl}(2n+1)$, $n\ge 2$;

(3) $\mathfrak{so}(2n+1)\subset \mathfrak{so}(2n+2)$, $n\ge 3$;

(4) $G_2\subset \mathfrak{so}(7)$; 

(5) $G_2\subset \mathfrak{so}(8)$; 

(6) $G_2\subset \mathfrak{sl}(7)$; 

(7) $F_4\subset E_6$. 

(8) ${\mathfrak{sl}_2}\subset \g$ for any simple $\g$. 

Namely, the subalgebras (1),(2),(3),(5),(7) are obtained as fixed points of a Dynkin diagram automorphism, 
(4) is obtained by composing (5) and (3), and (6) is obtained by composing (5) and (2).
\end{theorem} 

Note that Theorem \ref{SSthm} holds not only in characteristic zero but also in sufficiently large characteristic (for each fixed $\g$). 

\begin{question} Suppose that the groups $K\subsetneq G$ are fixed. Is it true that for large enough $p$, all surjective tensor functors  
$F: \Ver_p(G)\to \Ver_p(K)$ are given by Corollary \ref{funsemisi} (up to autoequivalences of $\Ver_p(G)$ and $\Ver_p(K)$)? 
\end{question} 
 
\section{Objects of finite type in semisimplifications}

Let $\D$ be a semisimple tensor category and $X\in \D$. Let us say that $X$ is of finite type if the number of isomorphism classes of simple objects occurring in tensor products of $X$ and $X^*$ is finite; i.e., $X$ generates a fusion subcategory $\D_X\subset \D$. If $\overline{\C}$ is the semisimplification of a category $\C$, and $X\in \C$, we will say that $X$ is of finite type if so is $\overline{X}$. It is an interesting question which objects of $\C$ are of finite type. 
Note that according to Example \ref{ex1}(4), $X$ does not have to be of finite type even if $\C$ is the representation category of a finite group (e.g.
$\C=\Rep_\k (\Bbb Z/2)^2$ for ${\rm char}(\k)=2$).  

Yet, a lot of interesting representations of finite groups do turn out to be of finite type, and generate interesting fusion categories. 
The goal of this subsection is to give some examples of such representations. 

Let $H$ be an affine algebraic group over an algebraically closed field $\k$ of characteristic zero. Let $V$ be a rational representation of $H$. 
 Let $H_V$ be the reductive envelope of $H$ inside $GL(V)$ defined in Definition \ref{aksec}. Assume that $H$ contains a regular unipotent element of $H_V$. (e.g. $H=U_n$, the maximal unipotent subgroup of $SL(n)$ and $V=\k^n$; then $H_V=SL(V)$). Note that all this data is defined over some finitely generated subring $R\subset \k$, hence can be reduced modulo $p$ for sufficiently large $p$; namely, given a homomorphism $\psi: R\to \overline{\Bbb F_p}$, 
we have $\psi(R)=\Bbb F_q$, where $q=p^r$ for some $r$, and we have a chain of finite groups 
$H(\Bbb F_q)\subset H_V(\Bbb F_q)\subset GL(V(\Bbb F_q))$. Let $V_\psi=V(\overline{\Bbb F}_p)$; it is a representation of 
these finite groups over $\overline{\Bbb F_p}$. Let $\C:=\Rep_{\overline{\Bbb F_p}}H(\Bbb F_q)$.

\begin{theorem}\label{finty} For large enough $p$, the category $\overline{\C}_{\overline{V_\psi}}$ generated by $\overline{V_\psi}$ is a quotient of $\Ver_p(H_V)=\Ver_p(H_V,\overline{\Bbb F_p})$. In particular, the object $V_\psi$ is of finite type in $\C$. 
\end{theorem}    
  
\begin{proof} 
We have an additive monoidal restriction functor 
$$
{\rm Res}: {\rm Tilt}(H_V(\overline{\Bbb F_p}))\to \Rep_{\overline{\Bbb F}_p} H(\Bbb F_q),
$$ 
hence an additive monoidal functor 
$$
\bold S\circ {\rm Res}:  {\rm Tilt}(H_V(\overline{\Bbb F_p}))\to \overline{\Rep_{\overline{\Bbb F}_p} H(\Bbb F_q)}.
$$ 
Moreover,
the image of a negligible module under the functor ${\rm Res}$ is negligible, as it is already so
after restricting to the group $\Bbb Z/p$ generated by a regular unipotent element of $H_V$ contained in $H(\Bbb F_q)$ (\cite{J}, E13). 
Hence the functor $\bold S\circ {\rm Res}$
descends to a tensor functor
 $\widetilde F: {\rm Ver}_p(H_V)\to \overline{\Rep_{\overline{\Bbb F}_p} H(\Bbb F_q)}$
 (this functor is automatrically exact since the source category is semisimple). 
Moreover, the functor $\widetilde F$ lands in $\overline{\C}_{\overline{V_\psi}}$, so we get a surjective
tensor functor $F: {\rm Ver}_p(H_V)\to \overline{\C}_{\overline{V_\psi}}$. 
In particular, in this case $\overline{\C}_{\overline{V_\psi}}$ is 
a quotient of ${\rm Ver}_p(H_V)$, thus a fusion category if $H_V$ is semisimple.

Moreover, even if $H_V$ is not semisimple but only reductive, $\overline{\C}_{\overline{V_\psi}}$ is still a fusion category, since 1-dimensional representations of $H_V$ obviously have finite order when restricted to the finite group $H(\Bbb F_q)$. 
\end{proof} 

\begin{conjecture} For sufficiently large $p$ the surjective tensor functor $F: {\rm Ver}_p(H_V)\to \overline{\C}_{\overline{V_\psi}}$ is an equivalence. 
\end{conjecture} 

\begin{remark} Let $\C$ be a symmetric tensor category over a field $\k$ of characteristic $p>0$, $\overline{\C}$ be its semisimplification, and $X\in \C$.
According to Conjecture 1.3 of \cite{O}, there should be a Verlinde fiber functor $F: \overline{\C}_{\overline{X}}\to \Ver_p$ (this is actually a theorem if $X$ is of finite type, see \cite{O}). 
So, in particular, assuming this conjecture, we can define the number $d(X):={\rm FPdim}(F(\overline{X}))$, the Frobenius-Perron dimension of $F(\overline{X})$. A more refined invariant is the full decomposition of $F(\overline{X})$ into the simple objects $L_1,...,L_{p-1}$ 
of $\Ver_p$: $F(\overline{X})=\sum_i a_i(X)L_i$. It is an interesting question how to compute these invariants for a given $X$ (actually, this question can also be asked in characteristic zero, with $\Ver_p$ replaced by ${\rm Supervec}$). Also, one can define the affine group scheme $G_X={\rm Aut}(F)$ in $\Ver_p$ (or ${\rm Supervec}$), and its dimension $\delta(X)$ 
is another interesting invariant of $X$. Note that $X$ is of finite type if and only if $\delta(X)=0$. 
Also note that if $X=V_\psi$ in the setting of Theorem \ref{finty}, then the above invariants can be easily computed using 
the results of \cite{EOV}. 
\end{remark} 

We also have the following proposition. 

\begin{proposition} Let $G$ be a finite group and $V$ a representation of $G$ over an algebraically closed field $\k$ of characteristic $p$ of dimension $d<p$. Suppose that there exists an element $g\in G$ such that the restriction of $V$ to the cyclic group generated by $g$ is indecomposable. Then $V$ is of finite type.
\end{proposition} 

\begin{proof} We may assume that $V$ is faithful, i.e., $G\subset GL(d)$. Let $u$ be the unipotent part of $g$. Then $u$ is a power of $g$ and a regular unipotent element of $GL(d)$ (as it acts indecomposably on $V$). Hence the restriction functor ${\rm Tilt}_p(GL(d))\to {\rm Rep}_{\k}(G)$ descends to a tensor functor ${\rm Ver}_p(GL(d))\to \overline{{\rm Rep}_{\k}(G)}$. In particular, the tensor category generated by $V$ in the semisimplification of ${\rm Rep}_{\k}(G)$ is finite, as desired.  
\end{proof} 

This proposition can be generalized as follows, with a similar proof: 

\begin{proposition} Let $G$ be a finite group and $V$ a faithful representation of $G$ over an algebraically closed field $\k$ of characteristic $p$ of dimension $d<p$. Suppose that $K\subset GL(d)$ is a reductive subgroup containing $G$, such that $G$ contains a regular unipotent element of $K$. Then $V$ is of finite type. 
\end{proposition} 

\section{Semisimplification of ${\rm Tilt}(GL(n))$ when ${\rm char}(\k)=2$}

In this section we describe the category $\overline{\C}$ when $\C={\rm Tilt}(GL(n))$ and ${\rm char}(\k)=2$. 

First recall Lucas' theorem in elementary number theory: 

Let $a\in \Bbb Z, b\in \Bbb Z_+$ with $p$-adic expansions 
$$
a=\sum_i a_ip^i,\ b=\sum_i b_ip^i,
$$
with $0\le a_i,b_i\le p-1$. 

\begin{proposition}(Lucas' theorem) One has 
$$
\binom{a}{b}=\prod_i \binom{a_i}{b_i}\text{ mod }p.
$$
In particular, $\binom{a}{b}$ is not divisible by $p$ if and only if $b_i\le a_i$ for all $i$. 
\end{proposition} 

Let $V=\k^n$ be the tautological representation of $GL(n)$. 
Recall that the indecomposable objects of the category $\C$ are the indecomposable direct summands 
in tensor products of the fundamental modules $\wedge^\ell V$, $1\le \ell\le n$. Moreover, it is well known that we can take $\ell$ 
to be only powers of $2$. Indeed, if $\ell=2^{k_1}+\dots +2^{k_r}$ with $0\le k_1<\dots<k_r$ is the binary expansion of $\ell$
then by Lucas' theorem the multinomial coefficient 
$$
N:=\frac{\ell!}{2^{k_1}!\dots 2^{k_r}!}=\binom{\ell}{2^{k_1}}\binom{\ell-2^{k_1}}{2^{k_2}}...
$$
is odd. Now pick a subset of coset representatives $C\subset S_\ell$ mapping bijectively onto the quotient $S_\ell/(S_{2^{k_1}}\times\dots \times S_{2^{k_r}})$, 
and define the operator $P:=\sum_{g\in C}g$ on the space $\wedge^{2^{k_1}}V\otimes\dots \otimes \wedge^{2^{k_s}}V$. Since $|C|=N$, it is easy to see 
(e.g., by picking a basis of $V$) that $P^2=P$ and ${\rm Im}(P)=\wedge^\ell V$, which shows that $\wedge^\ell V$ is naturally a direct summand 
in $\wedge^{2^{k_1}}V\otimes\dots \otimes \wedge^{2^{k_s}}V$, as desired. 

This shows that the semisimplification $\overline{\C}$ is generated 
by the objects $X_m:=\overline{\wedge^{2^m}V}$, with $0\le m\le \log_2n$. 
Note that $\dim_\k X_m=\binom{n}{2^r}$, which by Lucas' theorem is odd if and only if the $m$-th digit (from the right) in the binary expansion of $n$ is $1$. Thus we can keep only $X_m$ with such values of $m$. 
In other words, $\overline{\C}$ 
is generated by $X_{m_1},...,X_{m_s}$, where $n=2^{m_1}+\dots+2^{m_s}$, $0\le m_1<\dots<m_s$, is the binary expansion of $n$.  

\begin{proposition}\label{inverti}
Let $n_j=2^{m_1}+\dots+2^{m_j}$, $1\le j\le s$, and let
$Y_j:=\overline{\wedge^{n_j}V}$. Then $Y_j$ is invertible. Moreover, 
$Y_j=X_{m_j}\otimes Y_{j-1}$ (where $Y_0:=\bold 1$), so $X_{m_j}$ are invertible as well. 
Hence the category $\overline{\C}$ is pointed. 
\end{proposition} 

\begin{proof} To prove that $Y_j$ is invertible, it suffices to show that the module $\wedge^{n_j}V\otimes (\wedge^{n_j}V)^*$ has a unique indecomposable direct summand of odd dimension, namely $\bold 1$ (which is a direct summand using the evaluation and coevaluation maps). To this end, it suffices to show that this is so after restriction of this representation to any subgroup $G\subset GL(n)$. Take $G=GL(n_j)\times GL(n-n_j)$. 
Then $V=V'\oplus V''$, where $\dim_\k V'=n_j$ and $\dim_\k V''=n-n_j$, and 
$$
\wedge^{n_j}V|_G=\oplus_{i=0}^{n_j}\wedge^{n_j-i}V'\otimes \wedge^iV''.
$$
Note that $n-n_j$ is divisible by $2^{m_j+1}>n_j$, hence by Lucas' theorem $\wedge^iV''$ is even dimensional for any $0<i\le n_j$. This means that any odd-dimensional indecomposable direct summand in 
the $G$-module \linebreak $\wedge^{n_j}V\otimes (\wedge^{n_j}V)^*$ is $\wedge^{n_j}V'\otimes (\wedge^{n_j}V')^*=\bold 1$. 

Now, $Y_j$ is a direct summand in $X_{m_j}\otimes Y_{j-1}$. Since $X_{m_j}$ is simple and  $Y_{j-1}$ is invertible, we get $Y_j=X_{m_j}\otimes Y_{j-1}$, i.e., $X_{m_j}=Y_j\otimes Y_{j-1}^*$ is invertible.  
\end{proof} 

\begin{proposition}\label{indep} The objects $X_{m_j}$, $j=1,\dots,s$ are multiplicatively independent. 
In other words, we have $\overline{\C}=\Vec_\k(\Bbb Z^s)$, where the group $\Bbb Z^s$ 
is generated by the isomorphism classes of the objects $X_{m_j}$ (or $Y_j$). 
\end{proposition} 

\begin{proof} Assume the contrary, i.e., that we have a nontrivial relation 
\begin{equation}\label{ide}
X_{m_1}^{\otimes p_1}\otimes\dots \otimes X_{m_\ell}^{\otimes p_\ell}\cong
X_{m_1}^{\otimes q_1}\otimes\dots \otimes X_{m_\ell}^{\otimes q_\ell},
\end{equation}
where $\ell\le s$ and $p_i,q_i\in \Bbb Z_{\ge 0}$ 
with $p_iq_i=0$ for $1\le i\le \ell$ and $p_\ell\ne 0$ (so $q_\ell=0$). 
Let $r=2^{m_\ell}+\dots+2^{m_s}$, so that 
$n-r=2^{m_1}+\dots+2^{m_{\ell-1}}$. Consider the subgroup $G=GL(r)\times GL(n-r)\subset GL(n)$. Then $V=V'\oplus V''$, where 
$\dim_\k V'=r$ and $\dim_\k V''=n-r$. For each $1\le j\le \ell$, we have 
$$
\wedge^{2^{m_j}}V|_G=\oplus_{i=0}^{2^{m_j}}\wedge^i V'\otimes \wedge^{2^{m_j}-i}V'',
$$
Since $r$ is divisible by $2^{m_\ell}$ and $n-r<2^{m_\ell}$, all the indecomposable summands in this direct sum have even dimension except $i=0$ for $j<l$ and $i=2^{m_\ell}$ for $j=\ell$. 
Thus, the only odd-dimensional indecomposable direct summand of $\wedge^{2^{m_j}}V|_G$ 
is a trivial representation of $GL(r)$ except for $j=\ell$, in which case $GL(r)$ acts 
on this summand by the determinant character. Thus, $GL(r)$ acts trivially on the unique odd-dimensional indecomposable direct summand 
on the right hand side of \eqref{ide} but by $\det^{p_\ell}$ on such summand the left hand side, which is a contradiction. 
\end{proof} 

\begin{corollary} We have $\overline{{\rm Tilt}(SL_n)}=\Vec_\k(\Bbb Z^{s-1})$ where $\Bbb Z^{s-1}$ is generated by 
the isomorphism classes of $X_{m_1},\dots,X_{m_s}$ modulo the relation $X_{m_1}\dots X_{m_s}=1$, and $\overline{{\rm Tilt}(PGL_n)}=\Vec_\k(\Bbb Z^{s-1})$
where $\Bbb Z^{s-1}$ is the group of $X_{m_1}^{\otimes n_1}\dots X_{m_s}^{\otimes n_s}$, $n_1\dots,n_s\in \Bbb Z$, where $\sum_i 2^{m_i}n_i=0$. 
\end{corollary} 

\begin{proof} Straightforward from Proposition \ref{indep}. 
\end{proof} 

\appendix
\section{Categorifications of based rings attached to $SO(3)$.}
The goal of this Appendix is to deduce some classification results on categorifications
of certain based rings from the results of \cite{MPS}.
We assume that the base field $\k$ is algebraically closed of characteristic zero. 

\subsection{}
We consider the based ring $K_\infty$ (see \cite[Chapter 3]{EGNO}) with basis 
$X_i, i\in {\mathbb Z}_{\ge 0}$ and with multiplication determined by
$$X_0=1,\; X_1X_i=X_iX_1=X_{i-1}+X_i+X_{i+1},\; i\ge 1.$$
It is a classical fact that $K_\infty$ is isomorphic to the representation ring of the group
$SO(3)$ via the map sending $X_i$ to a unique irreducible representation of dimension
$2i+1$.

We will consider {\em pivotal categorifications} of $K_\infty$, that is, semisimple pivotal tensor categories
$\C$ equipped with an isomorphism of based rings $K(\C)\simeq K_\infty$ (cf. \cite[4.10]{EGNO}).
Any such category $\C$ is automatically spherical since every object of $\C$ is self-dual.
Let $X\in \C$ be an object such that its class $[X]$ corresponds to $X_1\in K_\infty$.
Let $d\in \k$ be the dimension of $X$. There exists $\q\in \k$ such that 
$d=[3]_\q=\q^2+1+\q^{-2}$.

\begin{theorem} \label{mps}
(i) Assume that $\q^2=1$ or that $\q^2$ is not a root of 1.
Then $\C$ is equivalent to the category $\Rep(SO(3)_\q)$ (see \cite[Section 4]{MPS}).

(ii) Assume $\q^2=-1$. Then $\C$ is equivalent to the category $\Rep(OSp(1|2))$ (see \cite[Section 4]{MPS}).
\end{theorem}

\begin{proof} Let $\C_0$ be the monoidal subcategory of $\C$ generated by $X$ and by (nonzero)
morphisms $\be \to X\ot X, X\ot X\to \be, X\to X\ot X, X\ot X\to X$. Thus:
$$\mbox{objects of}\; \C_0=X^{\ot n},\; n\in {\mathbb Z}_{\ge 0},$$
$$\mbox{morphisms of}\; \C_0=\mbox{morphisms in}\; \C \; \mbox{which are linear combinations}$$
$$\mbox{of tensor products and compositions of the four morphisms above}.$$
\vskip .05in
Let $\N$ be the ideal of negligible morphisms in $\C_0$, and let $\tilde \C=\C_0/\N$ be the quotient.
Clearly
\begin{eqnarray}\label{less}
&\dim \Hom_{\tilde \C}(X^{\ot m},X^{\ot n})\le \dim \Hom_{\C_0}(X^{\ot m},X^{\ot n})&\\
\nonumber &\le \dim \Hom_{\C}(X^{\ot m},X^{\ot n}).&
\end{eqnarray}
The category $\tilde \C$ is an example of a (possibly twisted) {\em trivalent category}, as defined in \cite[Section 7]{MPS} (thus $\tilde \C$ satisfies the assumptions of \cite[Definition 2.1]{MPS}
except, possibly, the rotational invariance of the morphism $X\to X\ot X$). 
Moreover, the numbers $\dim \Hom_{\tilde \C}(\be, X^{\otimes k})$ are bounded by
the numbers $d_k=\dim \Hom_{\C}(\be, X^{\otimes k})$, which are easily computable using
the isomorphism $K(\C)\simeq K_\infty$. In particular, $d_k=1,0,1,1,3$ for $k=0,1,2,3,4$.
Since $d\ne 2$, \cite[Proposition 7.1]{MPS} implies that $\tilde \C$ is not twisted, that is,
$\tilde \C$ is a trivalent category in the sense of \cite[Definition 2.1]{MPS}. Thus by
\cite[Theorem A]{MPS}, $\tilde \C$ is equivalent to $\Rep(SO(3)_\q)$ or $\Rep(OSp(1|2))$;
in particular, the Grothendieck ring $K(\tilde \C)$ of (the Karoubian envelope of) $\tilde \C$ is isomorphic to $K_\infty =K(\C)$. Thus, 
the inequalities in \eqref{less} are, in fact, equalities,
and the category $\C$ is equivalent to the Karoubian envelope of $\tilde \C$. The result
follows.
\end{proof}

\begin{remark} (i) We expect that the assumption on $\q$ in Theorem \ref{mps} is automatically
satisfied, i.e., there is no categorification of $K_\infty$ where $\q^2\ne \pm 1$ is a root of 1.
Moreover, it seems likely that the assumption on pivotality of $\C$ can also be dropped.

(ii) D.~Copeland and H.~Wenzl recently obtained a classification of {\em ribbon categorifications} of
the based rings  $K(\Rep(SO(n)_\q))$ for any $n$. In particular this implies Theorem \ref{mps}
(and Theorem \ref{fmps} below) under an additional assumption that the category $\C$
is braided.
\end{remark}

\subsection{Fusion categories}
For an integer $l\ge 2$ we consider the based ring $K_l$ with basis $X_i,\; i=0, \ldots ,l$ and
with multiplication determined by
$$X_0=1,\; X_1X_i=X_{i-1}+X_i+X_{i+1},\; i=1, \ldots l-1,\; X_1X_l=X_{l-1}.$$
The ring $K_l$ can be considered as a truncated version of the ring $K_\infty$. 
It is well known that the ring $K_l$ has categorifications of the form $\Rep(SO(3)_\q)=\Ver_\q^+$,
where $\q$ is a suitable root of 1. 

\begin{theorem} \label{fmps}
Let $\C$ be a pivotal fusion category which is a categorification of $K_l$
where $l>2$. Then there is a tensor equivalence $\C \simeq \Rep(SO(3)_\q)$ where $\q$
is a primitive root of 1 of degree $4(l+1)$.
\end{theorem}

\begin{proof} We start by classifying homomorphisms $\phi: K_l\to \k$. Any such homomorphism
is uniquely determined by $\phi(X_1)$; if $\phi_\q(X_1)=[3]_\q=\q^2+1+\q^{-2}$ then
$\phi_\q(X_i)=[2i+1]_\q=\frac{\q^{2i+1}-\q^{-2i-1}}{\q-\q^{-1}}$; in particular, the existence of $\phi_\q$ is equivalent to the equation
$$[3]_\q[2l+1]_\q=[2l-1]_\q\; \Longleftrightarrow \q^{2(l+1)}=\pm 1,\; \q^2\ne 1.$$ 
Clearly $\phi_\q=\phi_{\q'}$ if and only if $\q^2=\q'^{\pm 2}$. One computes easily
the {\em formal codegree} $f_{\phi_\q}$ (see e.g. \cite[Section 2.3]{Or3}) of $\phi_\q$:
$$f_{\phi_\q}=\left\{
\begin{array}{cc}
l+1&\mbox{if}\; \q^2=-1,\\
-\frac{2(l+1)}{(\q-\q^{-1})^2}&\mbox{if}\; \q^2\ne -1.
\end{array}\right.
$$

The category $\C$ is spherical, as all its objects are self-dual. Hence, by \cite[Corollary 2.15]{Or3},
the {\em dimension field} (i.e., the subfield of $\k$ generated by the dimensions of the objects) of $\C$ 
contains all $f_{\phi_\q}$; thus, the degree of the dimension field over the rationals is $\ge \frac12\varphi(2(l+1))$,
where $\varphi$ is the Euler function (this is the degree of the field generated by $\q^2+\q^{-2}$, where
$\q^2$ is a primitive root of 1 of degree $2(l+1)$). It follows that the dimension
homomorphism $K_l=K(\C)\to \k$ is $\phi_\q$, with $\q^2$ being  a primitive root of 1 of degree
$r$, where $r$ divides $2(l+1)$ and $\varphi(r)=\varphi(2(l+1))$. 
Thus, either $r=2(l+1)$ or $r=l+1$.
The latter case is possible only if $l+1$ is odd, and in this case 
$\phi_\q(X_{l/2})=0$, so $\phi_\q$ cannot be the dimension homomorphism.

Thus, we have proved that the dimension homomorphism $$K_l=K(\C)\to \k$$ coincides with
the dimension homomorphism $$K_l=K(\Rep(SO(3)_\q)\to \k,$$ where $\q$ is a primitive
root of 1 of degree $4(l+1)$. 

The rest of the proof is parallel to the proof of Theorem \ref{mps}.
We consider the subcategory $\C_0$ of $\C$ generated by the morphisms 
$$\be \to X\ot X, X\ot X\to \be, X\to X\ot X, X\ot X\to X$$ and its quotient $\tilde \C$ by negligible
morphisms. Then one deduces from \cite[Theorem A]{MPS} that the
(Karoubian envelope of) the category $\tilde \C$ is equivalent to $\Rep(SO(3)_\q)$, which
has the same Grothendieck ring as $\C$. This implies that $\tilde \C\cong\C$, and
the result follows.
\end{proof}

\begin{remark}\label{anoth} (i) The categorifications of $K_l$ with $l=2$ are completely classified in
\cite{EGO}. This case is somewhat different from the case $l>2$, see \cite[Section 7]{MPS}.

(ii) It is conjectured that any fusion category has a pivotal structure. Thus we expect
that the pivotality assumption in Theorem \ref{fmps} is superfluous. 

(iii) Another family of truncations of the ring $K_\infty$ is given by rings $\tilde K_l, \; l\ge 1$
with basis $X_0, \ldots, X_l$ and with multiplication
$$X_0=1,\; X_1X_i=X_{i-1}+X_i+X_{i+1},\; i=1, \ldots l-1,\; X_1X_l=X_{l-1}+X_l.$$
Such rings are also categorified by $\Rep(SO(3)_\q)$ where $\q$ is a suitable root of 1. 
It is easy to see that there are no other categorifications $\C$ of $\tilde K_l$, since 
$\C \boxtimes \Vec_{{\mathbb Z}/2{\mathbb Z}}$ would have been an example of a Temperley-Lieb category
generated by the object $X_l\boxtimes \bold 1$.
\end{remark}

\end{document}